\documentclass[11pt,reqno]{amsart}
 \usepackage[foot]{amsaddr}
\usepackage[cm]{fullpage}
\usepackage{mathrsfs,amssymb,graphicx,verbatim,amsmath,amsfonts}
\usepackage{paralist}
\usepackage[breaklinks,pdfstartview=FitH]{hyperref}
\usepackage{upgreek}
\usepackage{tikz}
\usepackage[mathscr]{euscript}
\usepackage[T1]{fontenc}
\usepackage{dutchcal}
\usepackage{fourier}
\usepackage{multirow}
\usepackage{caption}

\renewenvironment{quote}{%
   \list{}{%
     \leftmargin0.597cm   
     \rightmargin\leftmargin
   }
   \item\relax
}
{\endlist}

\renewcommand{\le}{\leqslant}
\renewcommand{\ge}{\geqslant}
\renewcommand{\leq}{\leqslant}
\renewcommand{\geq}{\geqslant}
\renewcommand{\setminus}{\smallsetminus}
\allowdisplaybreaks
\newcommand{\ud}[0]{\,\mathrm{d}}
\usepackage{esint}
\usepackage[autostyle]{csquotes}

\newcommand{\n}{\{1,\ldots,n\}}

\renewcommand{\d}{\delta}

\newcommand{\e}{\varepsilon}
\newcommand{\R}{\mathbb R}

\newtheorem{theorem}{Theorem}
\newtheorem{lemma}[theorem]{Lemma}
\newtheorem{proposition}[theorem]{Proposition}

\newtheorem{corollary}[theorem]{Corollary}

\newcommand{\Id}{\mathsf{Id}}
\theoremstyle{remark}
\newtheorem{remark}[theorem]{Remark}

\newtheorem{question}[theorem]{Question}

\renewcommand{\i}{\mathsf{i}}

\renewcommand{\subset}{\subseteq}
\newcommand{\dist}{\mathrm{dist}}

\newcommand{\N}{\mathbb N}
\newcommand{\Z}{\mathbb Z}

\newcommand{\eqdef}{\stackrel{\mathrm{def}}{=}}

\newcommand{\cc}{\mathsf{c}}

\renewcommand{\emptyset}{\varnothing}

\newcommand{\G}{\mathsf{G}}

\newcommand*\diff{\mathop{}\!\mathrm{d}}

\newcommand{\ms}{\mathscr}
\newcommand{\msf}{\mathsf}

\begin{document}

\title{Nonpositive curvature is not coarsely universal}

\author{Alexandros Eskenazis}
\address{(A.E.) Mathematics Department\\ Princeton University\\ Fine Hall, Washington Road, Princeton, NJ 08544-1000, USA}
\email{ae3@math.princeton.edu}

\author{Manor Mendel}
\address{(M.M.) Mathematics and Computer Science Department\\ The Open University of Israel\\ 1 University Road\\ P.O. Box 808\\ Raanana 43107\\ Israel}
\email{manorme@openu.ac.il}

\author{Assaf Naor}
\address{(A.N.) Mathematics Department\\ Princeton University\\ Fine Hall, Washington Road, Princeton, NJ 08544-1000, USA}
\email{naor@math.princeton.edu}

\thanks{A.E. was supported by BSF grant 2010021 and the Simons Foundation. M.M. was supported by BSF grant 2010021. A.N. was supported by NSF grant CCF-1412958, BSF grant 2010021, the Packard Foundation and the Simons Foundation. This work  was carried out under the auspices of the Simons Algorithms and Geometry (A\&G) Think Tank, and was completed while A.N. was a member of the Institute for Advanced Study.}



\begin{abstract}

We prove that not every  metric space embeds coarsely into an Alexandrov space of nonpositive curvature.
This answers a question of Gromov (1993) and is in contrast to the fact that any metric space embeds coarsely into an Alexandrov space of nonnegative curvature, as shown by Andoni, Naor and Neiman (2015). We establish this statement by proving that  a metric space which is $q$-barycentric for some $q\in [1,\infty)$ has metric cotype $q$ with sharp scaling parameter. Our proof utilizes nonlinear (metric space-valued) martingale inequalities and yields sharp bounds even for some classical Banach spaces.
This allows us to evaluate  the  bi-Lipschitz distortion of the $\ell_\infty$ grid $[m]_\infty^n=(\{1,\ldots,m\}^n,\|\cdot\|_\infty)$ into $\ell_q$ for all $q\in (2,\infty)$,  from which we deduce the following discrete converse to the fact that $\ell_\infty^n$ embeds with distortion $O(1)$ into $\ell_q$ for $q=O(\log n)$.
A  rigidity theorem of Ribe (1976) implies that for every $n\in \N$ there exists $m\in \N$ such that if $[m]_\infty^n$ embeds into $\ell_q$ with distortion $O(1)$, then $q$ is necessarily at least a universal constant multiple of $\log n$. Ribe's theorem does not give an explicit upper bound on this $m$, but by the work of Bourgain  (1987) it suffices to take $m=n$, and this was the previously best-known estimate for $m$. We show that the above discretization statement actually holds when $m$ is a universal constant.
\end{abstract}

\maketitle

\setcounter{tocdepth}{5}

\section{Introduction}

A complete geodesic metric space $(X,d_X)$ is  an Alexandrov space of nonpositive curvature if for  any quadruple of points $x,y,z,w\in X$ such that $w$ is a metric midpoint of $x$ and $y$, i.e., $d_X(w,x)=d_X(w,y)=\frac12d_X(x,y)$, we have
\begin{equation}\label{eq:def cat0}
d_X(z,w)^2+\frac14 d_X(x,y)^2 \le \frac12 d_X(z,x)^2+\frac12 d_X(z,y)^2.
\end{equation}
If the reverse inequality to~\eqref{eq:def cat0} holds true for any such quadruple $x,y,z,w\in X$, then $X$ is  an Alexandrov space of nonnegative curvature.\footnote{Strictly speaking, the above definition is of Alexandrov spaces of {\em global} nonpositive curvature, also known as $\mathrm{CAT}(0)$ spaces or Hadamard spaces. See~\cite{BH99} for the local counterpart of this definition; we will not treat it here, and therefore it will be convenient to drop the term "global" throughout the present text, because the ensuing results are vacuously false under the weaker local assumption (since the $1$-dimensional simplicial complex that is associated to any connected combinatorial graph is an Alexandrov space of local nonpositive curvature). For nonnegative curvature, the local and global notions coincide due to  the (metric version of the) Alexandrov--Toponogov theorem; see e.g.~\cite{Pla02}.}
See e.g.~\cite{BGP92,Per95,BH99,Stu99,BBI01,Stu03,Gro07} for  more on these fundamental notions.

A metric space $(Y,d_Y)$  is said to embed coarsely into a metric space $(X,d_X)$ if there exist two nondecreasing moduli $\omega,\Omega:[0,\infty)\to [0,\infty)$ satisfying $\omega\le \Omega$ pointwise and $\lim_{t\to\infty}\omega(t)=\infty$, and a mapping $f:Y\to X$ such that
\begin{equation}\label{eq:def coarse}
\forall\, x,y\in Y,\qquad \omega\big(d_Y(x,y)\big)\le d_X\big(f(x),f(y)\big)\le \Omega\big(d_Y(x,y)\big).
\end{equation}
A mapping $f:Y\to X$ that satisfies~\eqref{eq:def coarse} is called  a {coarse embedding} (with moduli $\omega,\Omega$).

The notion of a coarse embedding was introduced by Gromov in~\cite[\S4]{Gro88}, where such an embedding was called a "placement," and further studied by him in~\cite[Section~7.E]{Gro93}, where such an embedding was called a "uniform embedding." The subsequent change to the currently commonly used term "coarse embedding" is due to the need to avoid conflict with prior terminology in the functional analysis literature; see e.g.~the explanation in~\cite{BL00,Roe03,NY12,Ost13}.

As a special case of a more general result that will be described later, we will prove here the following theorem.

\begin{theorem}\label{thm:cat0 special case}
There is a metric space $Y$ that does not embed coarsely into any nonpositively curved Alexandrov space $X$.
\end{theorem}

Theorem~\ref{thm:cat0 special case} is the first time that the mere existence of such a metric space $Y$ is established (thus addressing a longstanding question of Gromov~\cite{Gro93}; see below), but we will actually see  that one could take here $Y=\ell_p$ for any $p>2$.

It follows from~\cite{ANN16} that the statement of Theorem~\ref{thm:cat0 special case} is false if one replaces in it the term "nonpositively curved" by "nonnegatively curved." Namely, by~\cite{ANN16} {every} metric space embeds coarsely into some nonnegatively curved Alexandrov space. This  difference between the coarse implications of the "sign" of curvature is discussed further in Section~\ref{sec:hierarchy}  below, where we  also explain how our work answers other open questions that were posed in~\cite{ANN16}. A conceptual contribution  that underlies Theorem~\ref{thm:cat0 special case}  is to specify an invariant which is preserved under  embeddings that may incur very large distortion, such that  this invariant holds when the curvature is nonpositive yet it does not follow from nonnegative curvature. Prior to this, nonnegative curvature was observed to be "better behaved" than nonpositive curvature, in the sense that  all such invariants that were previously computed for Alexandrov spaces   held/failed equally well in the presence of either nonpositive or nonnegative curvature (for different reasons), or they held for spaces of nonnegative curvature and not necessarily for spaces of nonpositive curvature; see Section~\ref{sec:table} below.

The geometric faithfulness that definition~\eqref{eq:def coarse} of a coarse embedding imposes is  weak (much more so than, say, that of a bi-Lipschitz embedding), but it nevertheless has strong implications in topology, $K$-theory and group theory; see e.g.~\cite{Gro93,FRR95,Yu00,GHW05,Yu06,NY12,Ost13}. The fact that the requirement~\eqref{eq:def coarse} is not stringent implies that large classes of metric spaces admit a coarse embedding into "nice" metric spaces (see~\cite{GY00,CCJV01,BG05,GHW05,Yu05} for examples of theorems of this type), and it  raises the question of finding invariants that  serve as obstructions  to the existence of coarse embeddings.

The above (open-ended) question was first raised in~\cite[page~218; Remark (b)]{Gro93}, where it is stated that

\smallskip

\begin{quote}{  \em "There is no known geometric obstruction for uniform embeddings into infinite dimensional spaces. In particular, it is unclear whether every separable metric space can be uniformly embedded into the Hilbert space $\R^\infty$."

\begin{flushright}{\em M.~Gromov, 1993.}\end{flushright}
}\end{quote}
\smallskip
Here, the term "infinite dimensional" alludes  to the fact that Gromov's investigations in~\cite{Gro88,Gro93} dealt with coarse embeddings into nonpositively curved spaces which exhibit a certain (appropriately defined)  finite dimensionality. In that setting, he indeed found examples of the sought-after obstructions (e.g.~relying on a coarse notion of dimension, or proving that the rank of a symmetric space serves as an obstruction for coarse embeddings of one symmetric space into another). The particular case of  embedding into a Hilbert space that is mentioned separately in the above quote is clearly the natural place to start, but it is also important since coarse embeddings  into a  Hilbert space have profound topological and $K$-theoretical implications, as conjectured later by Gromov~\cite{FRR95} and proven by Yu~\cite{Yu00}. Theorem~\ref{thm:cat0 special case} answers the  above question in the  setting in which it was originally posed, namely  ruling out the existence of  a coarse embedding into an arbitrary Alexandrov space of nonpositive curvature.

Gromov's  question influenced the development of several approaches which resolve its  Hilbertian  case (and variants for  embeddings into certain Banach spaces, partially motivated by an application in~\cite{KY06}), starting with the initial solution by Dranishnikov, Gong, Lafforgue, and Yu~\cite{DGLY02},  who were the first to prove that there exists a metric space which does not embed coarsely into a Hilbert space. The main threads in this line of research are

\begin{enumerate}
\item Dranishnikov, Gong, Lafforgue, and Yu~\cite{DGLY02} were the first to answer the Hilbertian case of Gromov's question, by adapting a classical argument of Enflo~\cite{Enf69-smirnov} based on an invariant which he called {\em generalized roundness}.
\item Gromov himself found~\cite{Gro00,Gro03} a different solution, showing that a metric space which contains any sequence of arbitrarily large bounded degree {\em expander graphs} fails to admit a coarse embedding into a Hilbert space.
\item Pestov~\cite{Pes08} proved that the universal Urysohn metric space $\mathbb{U}$ does not admit a coarse embedding into any uniformly convex Banach space, relying on works of Hrushovski~\cite{Hru92}, Solecki~\cite{Sol05} and Vershik~\cite{Ver08}.

\item A Fourier-analytic argument of Khot and Naor~\cite{KN06}  shows that for every $n\in \N$ one can choose a lattice $\Lambda_n\subset \R^n$ such that any metric space that contains the flat tori $\{\R^n/\Lambda_n\}_{n=1}^\infty$ does not embed coarsely into a Hilbert space.

\item  Johnson and Randrianarivony~\cite{JR06} proved that if $p>2$, then $\ell_p$ does not embed coarsely into a Hilbert space, by building on a method of Aharoni, Maurey and Mityagin~\cite{AMM85}. Randrianarivony~\cite{Ran06} proceeded to use this approach to characterize the Banach spaces that admit a coarse embedding into a Hilbert space as those that are linearly isomorphic to a closed subspace of an $L_0(\mu)$ space.
\item Mendel and Naor introduced~\cite{MN08}  an invariant called {\em metric cotype} and showed that it yields an obstruction to coarse embeddings provided that an auxiliary quantity called the "scaling parameter" has a sharp asymptotic behavior; all of the relevant terminology will be recalled in Section~\ref{sec:cotype} below, since this is the strategy of our proof of Theorem~\ref{thm:cat0 special case}. In~\cite{MN08}, such a sharp metric cotype inequality was established for $K$-convex Banach spaces via a vector-valued Fourier-analytic argument; here we take a different route in lieu of Fourier analysis,  due to its unavailability for functions that take values in metric spaces that are not Banach spaces.

\item Kalton investigated coarse embeddings in~\cite{Kal07} where, building in part on classical work of Raynaud~\cite{Ray83}, he introduced an invariant called {\em Kalton's Property $\mathcal{Q}$}, and used it to show that certain Banach spaces (including notably the space $c_0$ of null sequences, the James quasi-reflexive space~\cite{Jam50,Jam51}, and non-reflexive uniformly non-octahedral spaces~\cite{Jam74,JL75,Jam78,PX87}) do not admit a coarse embedding into any reflexive Banach space. Property $\mathcal{Q}$ can also be used to rule out the existence of coarse embeddings of certain Banach spaces into the Schatten--von Neumann trace class $\mathsf{S}_1$; see~\cite[page~172]{GMN11}. In the same work~\cite{Kal07}, Kalton considered a sequence $\{\mathsf{K}_r(\N)\}_{r=1}^\infty$ of infinite connected combinatorial graphs (equipped with their shortest-path metric), which are  called {\em Kalton's interlacing graphs}, and proved that any metric space which contains  $\{\mathsf{K}_r(\N)\}_{r=1}^\infty$ does not admit a coarse embedding into any {\em stable} metric space, hence a fortiori into a Hilbert space (here, the notion of stability of a metric space is in the sense of Garling's definition~\cite{Gar82} which builds on that of Krivine and Maurey~\cite{KM81}; see the survey~\cite{Bau14}, and we will return to this matter in Section~\ref{sec:open} below). This  circle of ideas has been substantially developed in several directions~\cite{KR08,Kal11,BLS18,LPP18,BLMS18}, leading in particular to coarse non-embeddability results also into some  non-classical Banach spaces, including the Tsirelson space~\cite{Tsi74}, the James quasi-reflexive space, and spaces that are reflexive and asymptotically $c_0$~\cite{MMT95}.
\item Lafforgue~\cite{Laf08,Laf09} and  Mendel and Naor~\cite{MN14,MN15} constructed a special type of (a sequence of)  expander graphs (namely, expanders relative to certain metric spaces, and even  {\em super-expanders}), which, via a straightforward generalization of the aforementioned argument of Gromov~\cite{Gro00,Gro03}, do not admit a coarse embedding into families of metric spaces that include Hilbert spaces but are much richer; we will describe this approach later as it relates to important open problems that pertain to a potential strengthening of Theorem~\ref{thm:cat0 special case}.
\item Arzhantseva and Tessera~\cite{AT15} introduced the notion of {\em relative expander graphs}, which is a weakening of the classical notion of expander graph, and showed that a metric space that contains any sequence of relative expanders does not admit a coarse embedding into a Hilbert space. They exhibited examples of such spaces into which no sequence of expander graph embeds coarsely. We will elaborate on this topic in Section~\ref{sec:open} below.
\item    Naor and Schechtman~\cite{NS16-KS} introduced an invariant called a {\em metric $\mathrm{KS}$ inequality}, evaluated it for some spaces (including $\ell_2$), and showed that it is an obstruction to coarse embeddings of powers of hypercubes.
\end{enumerate}

Section~\ref{sec:open} below returns to the above list of previous results/methods, explaining why they do not yield Theorem~\ref{thm:cat0 special case}. Briefly,  in some cases this is so because one can construct an Alexandrov space of nonpositive curvature for which the given approach fails. Some of the other methods in the above list rely so heavily on linear properties of the underlying Banach space that we do not see how to interpret those approaches when the target space is not a Banach space. And, for a couple of items in the above list, their possible applicability to the setting of Theorem~\ref{thm:cat0 special case} requires the (positive) solution of an open question that is of interest beyond its potential  use as a different route to Theorem~\ref{thm:cat0 special case}.

In the special case of embeddings into a simply connected Riemannian manifold $(M,d_M)$ of nonpositive sectional curvature, Gromov's question is (implicitly) answered by an argument of Wang~\cite{Wan98}, which implies that a metric space which contains any sequence of arbitrarily large bounded degree expander graphs fails to admit a coarse embedding into $(M,d_M)$; this also follows from the work of Izeki and  Nayatani~\cite{IN05}, and an explicit exposition of the proof appears in~\cite[pages~1159-1160]{NS11}. The assumption that $(M,d_M)$ is a Riemannian manifold can be relaxed to a technical requirement which restricts its possible singularities; in~\cite[Remark~1.2.C(b)]{Gro03} Gromov calls such Alexandrov spaces of nonpositive curvature "$\mathrm{CAT}(0)$ spaces with bounded singularities." Following this approach, several authors studied~\cite{IN05,Toy10,NS11,FT12,Nao14,Toy16} special cases of spaces $X$ for which the conclusion Theorem~\ref{thm:cat0 special case} holds true;  to indicate just one of several such examples that are available in the literature, Fujiwara and Toyoda~\cite[Corollary~1.9]{FT12} proved that this is so when $X$ is an arbitrary $\mathrm{CAT}(0)$ cube complex. However, in~\cite[page~187]{Gro01} and~\cite[page~117]{Gro03} Gromov himself proposed a construction of Alexandrov spaces of nonpositive curvature for which his approach fails, and this was carried out by Kondo~\cite{Kon12}. By~\cite{MN15},  there are metric spaces (namely, those that contain a  specially-constructed sequence of bounded degree expander graphs) that do not embed coarsely into most such "Gromov--Kondo spaces," but the proof of this statement in~\cite{MN15} relies on particular properties of the specific construction.


\subsection{Barycentric metric spaces} For a set $\Omega$, let $\ms{P}_{\Omega}^{<\infty}$ be the space of all finitely supported probability measures on $\Omega$. A function $\mathfrak{B}:\ms{P}_\Omega^{<\infty}\to \Omega$ is said to be a barycenter map if $\mathfrak{B}(\delta_x)=x$ for every $x\in \Omega$, where $\d_x$ is the point mass at $x$.

Following the terminology of~\cite{MN13}, if $q\in [1,\infty)$, then a metric space $(X,d_X)$ is said to be $q$-barycentric with constant $\beta\in (0,\infty)$ if there exists a barycenter map $\mathfrak{B}:\ms{P}_X^{<\infty}\to X$ such that every $x\in X$ and $\mu\in\ms{P}_X^{<\infty}$ satisfy
\begin{equation} \label{eq:pbarycentric}
d_X\big(\mathfrak{B}(\mu),x\big)^q + \frac{1}{\beta^q} \int_X d_X\big(\mathfrak{B}(\mu),y\big)^q \diff\mu(y) \leq \int_X d_X(x,y)^q\diff\mu(y).
\end{equation}
A metric space is said to be barycentric if it is $q$-barycentric for some $q\in [1,\infty)$. We imposed the restriction $q\ge 1$ above because it is mandated by the barycentric requirement~\eqref{eq:pbarycentric} (unless $X$ is a singleton). Moreover, if $(X,d_X)$ contains a geodesic segment, then~\eqref{eq:pbarycentric} implies that $q\ge 2$, though barycentric metric spaces need not necessarily contain any geodesic segment. See Section~\ref{sec:prelim}  below for a quick justification of these facts and further discussion.

It is well known (see e.g.~\cite[Lemma~4.1]{LPS00} or~\cite[Theorem~6.3]{Stu03}) that any Alexandrov space of nonpositive curvature is $2$-barycentric with constant $\beta=1$. Hence,  Theorem~\ref{thm:cat0 special case} is a special case of the following theorem.

\begin{theorem}\label{thm:main bary}
There exists a metric space $Z$ that does not embed coarsely into any barycentric metric space.
\end{theorem}
In fact, we will establish the following more refined version of Theorem~\ref{thm:main bary}.


\begin{theorem}\label{thm:main q-bary} If $p,q\in [1,\infty]$ satisfy $p>q$, then  $\ell_p$ does not embed coarsely into any $q$-barycentric metric space.
\end{theorem}

\subsection{Sharp metric cotype}\label{sec:cotype} Following~\cite{MN08}, a metric space $(X,d_X)$ is said to have metric cotype $q\in(0,\infty)$ with constant $\Gamma\in(0,\infty)$ if for every $n\in\N$ there exists some $m=m(n,q,X)\in\N$ such that every function $f:\Z_{2m}^n\to X$ satisfies
\begin{equation} \label{eq:cotype}
\Bigg(\sum_{i=1}^n \sum_{x\in\Z_{2m}^n} d_X\big( f(x+me_i),f(x)\big)^q\Bigg)^{\frac{1}{q}} \leq \Gamma m \Bigg(\frac{1}{2^n} \sum_{\e\in\{-1,1\}^n} \sum_{x\in\Z_{2m}^n} d_X\big(f(x+\e),f(x)\big)^q\Bigg)^{\frac{1}{q}}.
\end{equation}
In~\eqref{eq:cotype} and in what follows, $\Z_{2m}=\Z/(2m\Z)$ and the additions that appear in the arguments of $f$ are modulo $2m$. Also, $e_1=(1,0,\ldots,0),\ldots,e_n=(0,\ldots,0,1)$ is the standard basis of $\Z_{2m}^n$ (or, with a subsequent slight abuse of notation,  of $\R^n$).

See~\cite{MN08} for an explanation of this terminology, as well as its motivation within an extensive long-term research program called the {\em Ribe program}. Explaining this larger context is beyond the scope of the present article, but accessible introductory surveys are available~\cite{Kal08,Nao12,Bal13,Ost13,God17,Nao18}. It suffices to say here that our work is yet another example of an application of the Ribe program in which insights that are inspired by Banach space theory are used to answer a geometric question about objects which a priori have nothing to do with Banach spaces. As we shall see below, in addition to the aforementioned application, revisiting these matters in the fully nonlinear setting of Alexandrov spaces forced us to find a different approach that led to new  results even for  Banach spaces such as $\ell_q$.

\begin{remark}\label{rem:not quite same def}
The above definition of metric cotype $q$ is not identical to the definition of~\cite{MN08}. The difference is that in~\cite{MN08} the average in the right hand side of~\eqref{eq:cotype} is over the $3^n$ points $\e\in \{-1,0,1\}^n$ rather than over the $2^n$ points $\e\in \{-1,1\}^n$.  In the context of~\cite{MN08} it was more natural to average over the $\ell_\infty^n$ edges $\e\in \{-1,0,1\}^n$, while in the present context an average over the sign vectors $\e\in \{-1,1\}^n$ arises naturally. This nuance is irrelevant for the application to coarse embeddings. Nevertheless, in Section~\ref{rem:equivcotype} below we will prove that the above variant of the definition of metric cotype $q$  in fact coincides with the original definition of~\cite{MN08}.
\end{remark}

As formulated above (and in the literature), the notion of metric cotype $q$ suppresses the value of the so-called "scaling parameter" $m=m(n,q,X)$. Nevertheless, it was shown in~\cite{MN08} that obtaining a good upper bound on $m$ is important for certain applications, including as an obstruction to coarse embeddings. As explained in~\cite[Lemma~2.3]{MN08}, if $(X,d_X)$  is any non-singleton metric space that satisfies~\eqref{eq:cotype}, then necessarily $m\ge \frac{1}{\Gamma}n^{1/q}$. So, say that $(X,d_X)$ has \emph{sharp} metric cotype $q$ if $(X,d_X)$ has metric cotype $q$ with $m,\Gamma$ in~\eqref{eq:cotype} satisfying\footnote{In addition to the usual $"O(\cdot),o(\cdot)"$ asymptotic notation, it will be convenient to use throughout this article the following (also standard) asymptotic notation. Given two quantities $Q,Q'>0$, the notations
$Q\lesssim Q'$ and $Q'\gtrsim Q$ mean that $Q\le CQ'$ for some
universal constant $C>0$. The notation $Q\asymp Q'$
stands for $(Q\lesssim Q') \wedge  (Q'\lesssim Q)$. If  we need to allow for dependence on parameters, we indicate this by subscripts. For example, in the presence of  auxiliary objects (e.g.~numbers or spaces) $\phi,\mathfrak{Z}$, the notation $Q\lesssim_{\phi,\mathfrak{Z}} Q'$ means that $Q\le C(\phi,\mathfrak{Z})Q' $, where $C(\phi,\mathfrak{Z}) >0$ is allowed to depend only on $\phi,\mathfrak{Z}$; similarly for the notations $Q\gtrsim_{\phi,\mathfrak{Z}} Q'$ and $Q\asymp_{\phi,\mathfrak{Z}} Q'$.} the bounds $m\lesssim_{q,X} n^{1/q}$ and $\Gamma\lesssim_{q,X}1$.

Prior to the present work, the literature contained only one theorem which establishes that a certain class of metric spaces has sharp metric cotype $q$. Namely, by~\cite{MN08} this is so for $K$-convex Banach spaces of Rademacher cotype $q$ (see e.g.~\cite{Mau03} for the definitions of the relevant Banach space concepts; we will not use them in the ensuing proofs). The question whether every Banach space of Rademacher cotype $q$ has sharp metric cotype $q$ remains a fundamental open problem~\cite{MN08}. On the other hand, in~\cite[Theorem~1.5]{VW10} it was shown that for some $q\in [1,\infty)$, some classes of metric spaces (including ultrametrics) have metric cotype $q$ but do not have sharp metric cotype $q$; another such example appears in Remark~\ref{rem:r3} below. The following theorem yields a new setting in which sharp metric cotype holds.

\begin{theorem}[$q$-barycentric implies sharp metric cotype $q$] \label{thm:main} Fix $q,\beta\in [1,\infty)$ and let $(X,d_X)$ be a $q$-barycentric metric space with constant $\beta$. Then, for every $n\in\N$ and $m\in 2\N$, every function $f:\Z_{2m}^n\to X$ satisfies
\begin{equation} \label{eq:main}
\Bigg(\sum_{i=1}^n \sum_{x\in\Z_{2m}^n} d_X\big( f(x+me_i),f(x)\big)^q\Bigg)^{\frac{1}{q}} \le \Big(4n^{\frac{1}{q}}+\beta m\Big)\Bigg(\frac{1}{2^n} \sum_{\e\in\{-1,1\}^n} \sum_{x\in\Z_{2m}^n} d_X\big(f(x+\e),f(x)\big)^q\Bigg)^{\frac{1}{q}}.
\end{equation}
In particular, if $m\geq \frac{1}{\beta} n^{1/q}$, then $(X,d_X)$ satisfies the metric cotype $q$ inequality~\eqref{eq:cotype} with constant $\Gamma\lesssim \beta$.
\end{theorem}

\begin{remark} Beyond metric cotype, dimension-dependent scaling parameters occur (for conceptually distinct reasons) in other metric inequalities that arise in the Ribe program. Determining their asymptotically sharp values  is a major difficulty that pertains to important open problems; see~\cite{MN07,GN10,VW10,GMN11,NS16,Nao16,NS18,EMN18}. The currently best-known general bound~\cite{GMN11}   for the metric cotype $q$ scaling parameter of Banach spaces of Rademacher cotype $q$ is $m\lesssim n^{1+1/q}$.
\end{remark}

The deduction of Theorem~\ref{thm:main q-bary} (hence also its special cases Theorem~\ref{thm:main bary} and Theorem~\ref{thm:cat0 special case}) from Theorem~\ref{thm:main} follows from  an argument of~\cite{MN08}; for completeness, we will provide this derivation in Section~\ref{sec:deduce coarse} below. Beyond this  application, it turns out that Theorem~\ref{thm:main} sometimes yields new information even when the underlying metric space is a Banach space. To explain this, fix $q\in [2,\infty)$ and $\mathscr{K}\in [1,\infty)$. Following the terminology of~\cite{BCL94} (itself inspired by~\cite{Pis75,Fig76}; see also the treatment in~\cite{Pis16}), a Banach space $(X,\|\cdot\|_X)$ is said to be $q$-uniformly convex with constant $\mathscr{K}\in[1,\infty)$ if
\begin{equation} \label{eq:pconvex}
\forall\, x,y\in X,\qquad 2\|x\|_X^q + \frac{2}{\mathscr{K}^q} \|y\|_X^q \leq \|x+y\|_X^q+ \|x-y\|_X^q.
\end{equation}
The minimum $\mathscr{K}$ for which \eqref{eq:pconvex} holds is denoted $\mathscr{K}_q(X)$. As shown in~\cite[Lemma~3.1]{Bal92} (for $q=2$) and~\cite[Lemma~6.5]{MN14} (for general $q\ge 2$), every $q$-uniformly convex Banach space is also $q$-barycentric with constant  $\beta \le2\mathscr{K}_q(X)$. Combining this fact with Theorem \ref{thm:main}, we deduce the following statement
\begin{corollary} \label{cor:pconvex}
Suppose that $q\in[2,\infty)$ and let $(X,\|\cdot\|_X)$ be a $q$-uniformly convex Banach space. Then, for every $n\in\N$ and $m\in 2\N$, every function $f:\Z_{2m}^n\to X$ satisfies
\begin{equation} \label{eq:main-banach}
\Bigg(\sum_{i=1}^n \sum_{x\in\Z_{2m}^n} \big\| f(x+me_i)-f(x)\big\|_X^q\Bigg)^{\frac{1}{q}} \lesssim \Big(n^{\frac{1}{q}}+\mathscr{K}_q(X)m\Big)\Bigg(\frac{1}{2^n} \sum_{\e\in\{-1,1\}^n} \sum_{x\in\Z_{2m}^n} \big\|f(x+\e)-f(x)\big\|_X^q\Bigg)^{\frac{1}{q}}.
\end{equation}
\end{corollary}

In~\cite[Section~4]{MN08} a bound that is similar to~\eqref{eq:main-banach} was obtained under the assumption that $X$ is a $K$-convex Banach space of Rademacher cotype $q$, in which case in the   right hand side of~\eqref{eq:main-banach} the quantity $\mathscr{K}_q(X)$ is replaced in~\cite{MN08} by the product of the operator norm of the Rademacher projection on $L_q(\{-1,1\}^n;X)$ and the Rademacher cotype $q$ constant of $X$. These two results are incomparable, in the sense that there are Banach spaces $X$ for which~\eqref{eq:main-banach}  is stronger than the bound of~\cite{MN08}, and vice versa. To examine a concrete example, by the Clarkson inequality~\cite{Cla36}, if $q\in [2,\infty)$ and $X=\ell_q$, then we have $\mathscr{K}_q(X)=1$, so the first term in the right hand side of~\eqref{eq:main-banach} becomes $n^{1/q}+m$, which makes~\eqref{eq:main-banach} sharp in this case, up to the implicit absolute constant factor. In contrast, the Rademacher cotype $q$ constant of $\ell_q$ is equal to $1$  and, for sufficiently large $n=n(q)\in \N$, the  norm of the Rademacher projection on $L_q(\{-1,1\}^n;\ell_q)$ is at least a universal constant multiple of $\sqrt{q}$ (for a justification of the latter statement,  consider e.g.~\cite[Lemma~7.4.11]{Hyt17} with $N\asymp p$). Hence the corresponding term in the bound of~\cite{MN08} is $n^{1/q}+\sqrt{q}m$, which  is significantly weaker than~\eqref{eq:main-banach} if $q$ is large; we will describe in Section~\ref{sec:distrotion} below a geometric consequence of~\eqref{eq:main-banach} that relies on its behavior in the large $q$ regime and does not follow from its counterpart in~\cite{MN08}. On a more conceptual level, the fact that the Rademacher projection appears in the bound of~\cite{MN08} reflects the Fourier-analytic nature of its proof in~\cite{MN08}. In the present setting, we need an argument that works for functions that take values in barycentric metric spaces rather than Banach spaces, in which case we do not know how to interpret the considerations of~\cite{MN08}. The new route that we take here leads to the aforementioned better dependence on $q$ as $q\to \infty$ when $X=\ell_q$, though, as we already mentioned above, it is neither stronger nor weaker than the bound of~\cite{MN08} for general Banach spaces.

\subsection{Non-embeddability} Here we will derive some geometric consequences of Theorem~\ref{thm:main}, including  Theorem~\ref{thm:cat0 special case}.

\subsubsection{Coarse, uniform and quasisymmetric embeddings}\label{sec:deduce coarse} A metric space $(Y,d_Y)$  is said to embed uniformly into a metric space $(X,d_X)$ if there exists a one-to-one mapping $f:Y\to X$ such that both $f$ and $f^{-1}:f(Y)\to Y$ are uniformly continuous.  $(Y,d_Y)$  is said to embed quasisymmetrically into $(X,d_X)$ if there exists a one-to-one mapping $f:Y\to X$ and an increasing modulus $\eta:(0,\infty)\to (0,\infty)$ with $\lim_{t\to 0}\eta(t)=0$ such that for every distinct $x,y,z\in Y$ we have
 $$
 \frac{d_Y\big(f(x),f(y)\big)}{d_Y\big(f(x),f(z)\big)}\le \eta\left(\frac{d_X(x,y)}{d_X(x,z)}\right).
 $$
See~\cite{Isb64,TV80,BL00,Hei01,Ost13} and the references therein for (parts of) the large literature on these topics.

The proof of the following proposition is a straightforward abstraction of the arguments in~\cite{MN08} (for coarse and uniform embeddings) and~\cite{Nao12-quasi} (for quasisymmetric embeddings).

\begin{proposition}\label{prop:CUQ}
Suppose that $p,q\in [2,\infty]$ satisfy $p>q$. Then $\ell_p$ does not admit a coarse, uniform or quasisymmetric embedding into a metric space $(X,d_X)$ that has sharp metric cotype $q$. More generally, if a Banach space $(Y,\|\cdot\|_Y)$ admits a  coarse, uniform or quasisymmetric embedding into such $(X,d_X)$, then $Y$ has Rademacher cotype $q+\e$ for any $\e>0$.
\end{proposition}

For completeness, we shall now quickly prove Proposition~\ref{prop:CUQ} in the case of coarse embeddings, thus establishing that Theorem~\ref{thm:main} implies Theorem~\ref{thm:main q-bary}.

\begin{proof}[Proof of Proposition~\ref{prop:CUQ} in the case of coarse embeddings]Let $\omega,\Omega:[0,\infty)\to [0,\infty)$ satisfy $\omega\le \Omega$  and $\lim_{t\to \infty}\omega(t)=\infty$. Suppose that $(X,d_X)$ is a non-singleton metric space that satisfies the sharp metric cotype $q$ condition~\eqref{eq:cotype} for every $n\in \N$, some $\Gamma=\Gamma(q,X)>0$ and $m=m(n,q,X)\in \N$ obeying $m\le C n^{1/q}$ for some $C=C(q,X)>0$. By~\cite[Lemma~2.3]{MN08} we have $m\ge \frac{1}{\Gamma} n^{1/q}$.  Suppose also that $\phi:\ell_p^n\to X$ satisfies $\omega(d_X(x,y))\le \|\phi(x)-\phi(y)\|_p\le \Omega(d_X(x,y))$ for all $x,y\in \ell_p^n$.

Consider the function $f:\Z_{2m}^n\to X$ that is given by
$$
\forall\, x=(x_1,\ldots,x_n)\in \Z_{2m}^n,\qquad f(x)\eqdef \phi\left(n^{-\frac{1}{p}}me^{\frac{\pi\i}{m}x_1},\ldots,n^{-\frac{1}{p}}m e^{\frac{\pi\i}{m}x_n}\right).
$$
Then, for every $x\in \Z_{2m}^n$, every $j\in \n$ and every  $\e\in \{-1,1\}^n$ we have
$$
d_X\big(f(x+me_j),f(x)\big)\ge \omega \Big(n^{-\frac{1}{p}}m|e^{\pi \i}-1|\Big)\ge  \omega\Big(\frac{2}{\Gamma} n^{\frac{1}{q}-\frac{1}{p}}\Big) \qquad\mathrm{and}\qquad  d_X\big(f(x+\e),f(x)\big)\le \Omega\Big( m\big|e^{\frac{\pi\i}{m}}-1\big|\Big)\le \Omega(\pi).
$$
Due to these bounds, the assumed sharp metric cotype $q$ inequality  yields the estimate
\begin{equation}\label{eq:omega bounded}
\omega\Big(\frac{2}{\Gamma} n^{\frac{1}{q}-\frac{1}{p}}\Big) n^{\frac{1}{q}}\le \Gamma m\Omega(\pi) \le C \Omega(\pi)n^{\frac{1}{q}} \  \Longrightarrow \  \omega\Big(\frac{2}{\Gamma} n^{\frac{1}{q}-\frac{1}{p}}\Big)\le C\Omega(\pi).
\end{equation}
Since $q<p$, the validity of~\eqref{eq:omega bounded} for every $n\in \N$  contradicts the assumption that $\lim_{t\to \infty}\omega(t)=\infty$. The final statement of Proposition~\ref{prop:CUQ} follows by combining this conclusion with a (deep) classical theorem of Maurey and Pisier~\cite{MP76}.
\end{proof}

To state  one concrete  example of a locally finite metric space $Y$ for which the conclusion of Theorem~\ref{thm:cat0 special case} holds true, by using the above proof of Proposition~\ref{prop:CUQ} with the (arbitrarily chosen) value $p=3$, while recalling that by Theorem~\ref{thm:main}, when $X$ is an Alexandrov space of nonpositive curvature one can work with $m=\sqrt{n}$, shows that one can take
\begin{equation}\label{eq:disjoint union}
(Y,d_Y)=\bigsqcup_{m=1}^\infty \left(\sqrt[3]{m}\left\{e^{\frac{\pi \i}{m}},e^{\frac{2\pi \i}{m}},e^{\frac{3\pi \i}{m}},\ldots,e^{2\pi \i}\right\}^{m^2},\|\cdot\|_{\ell_3^{m^2}}\right).
\end{equation}
Here (and throughout), we use the standard convention that the disjoint union of bounded metric spaces is their set-theoretic disjoint union, equipped with the metric that coincides with the given metric within each "cluster," and the distance between points that belong to two different clusters is the maximum of the diameters of these two clusters.

The metric space in~\eqref{eq:disjoint union} is locally finite, but not of bounded geometry (see~\cite[Definition~1.66]{Ost13}).   We therefore ask

\begin{question}\label{Q:bounded geometry}
Does there exist a metric space of bounded geometry that does not embed coarsely into any Alexandrov space of nonpositive curvature? Does this hold true for coarse embedding into any barycentric space?
\end{question}
Given  that Theorem~\ref{thm:cat0 special case} has now been established, it seems plausible that the answer to Question~\ref{Q:bounded geometry} is positive. However, it isn't clear  whether the expander-based approach~\cite{Gro00,Gro03} that was previously used to address this matter for embeddings into a Hilbert space could be be applied here (certainly, by~\cite{Kon12}, one cannot use arbitrary expanders, as was done before. A tailor-made expander would be needed); see Section~\ref{sec:open} for more on this very interesting issue.

\subsubsection{Bi-Lipschitz distortion}\label{sec:distrotion} Suppose that $(U,d_U)$ and $(V,d_V)$ are non-singleton metric spaces, and that $f:U\to V$ is a one-to-one mapping. The (bi-Lipschitz) distortion of $f$, denoted $\dist(f)\in [1,\infty]$, is  the quantity
$$
\dist(f)\eqdef \sup_{\substack{x,y\in U\\ x\neq y}}\frac{d_V\big(f(x),f(y)\big)}{d_U(x,y)}\cdot  \sup_{\substack{a,b\in U\\ a\neq b}}\frac{d_U(a,b)}{d_V\big(f(a),f(b)\big)}.
$$
We shall also use the convention that $\dist(f)=\infty$ if $f$ is not one-to-one. The distortion of $U$ in  $V$, denoted $\cc_V(U)$, is  the infimum of $\dist(f)$ over all possible $f:U\to V$. In particular, by re-scaling we see that the distortion of a  metric space $(M,d_M)$ in a Banach space $(X,\|\cdot\|_X)$  is the infimum over those $D\in [1,\infty]$ for which there is $f:M\to X$ that satisfies
\begin{equation}\label{eq:def distortion}
\forall\, x,y\in M,\qquad \frac{d_M(x,y)}{D}\le \|f(x)-f(y)\|_X\le  d_M(x,y).
\end{equation}
For $p\in [1,\infty]$ and a finite metric space $(\Phi,d_\Phi)$, we will use below the simpler notation $\cc_{\ell_p}(\Phi)=\cc_p(\Phi)$.

Given $m,n\in \N$, let $[m]_\infty^n\subset \ell_\infty^n$ denote the finite grid $\{1,\ldots,m\}^n\subset \R^n$, equipped with the $\ell_\infty$ metric on $\R^n$. For each $q\in [1,\infty]$ we consider two "trivial" embeddings of $[m]_\infty^n$ into $\ell_q$. The first is the restriction of the identity mapping from $\ell_\infty^n$ to $\ell_q^n\subset \ell_q$, which we denote below by $\Id_{[m]_\infty^n\to \ell_q}$. The distortion of $\Id_{[m]_\infty^n\to \ell_q}$ equals $n^{1/q}$. For the second embedding, fix an arbitrary enumeration $\{x_1,\ldots,x_{m^n}\}$ of the $m^n$ elements of $[m]_\infty^n$, and let $\mathsf{Forget}_{[m]_\infty^n\to \ell_q}$ be the mapping
$$
\forall\, j\in \big\{1,\ldots,m^n\big\},\qquad \mathsf{Forget}_{[m]_\infty^n\to \ell_q}(x_j)=e_j\in \ell_q^{m^n}.
$$
The reason for the above notation/terminology is that this mapping  "forgets"  the metric structure of the grid $[m]_\infty^n$ altogether, as it is an arbitrary bijection of $[m]_\infty^n$ and the vertices of a simplex of $m^n$ vertices, on which the $\ell_q$ metric is trivial (equilateral). Since the diameter of $[m]_\infty^n$ equals $m$, the distortion of the embedding $\mathsf{Forget}_{[m]_\infty^n\to \ell_q}$ equals $m$.

The following consequence of Corollary~\ref{cor:pconvex} shows that when $q\ge 2$, the better of the  two trivial embeddings $\Id_{[m]_\infty^n\to \ell_q}$ and $\mathsf{Forget}_{[m]_\infty^n\to \ell_q}$ yields the smallest-possible (up to universal constant factors) distortion of the grid $[m]_\infty^n$ into $\ell_q$.

\begin{corollary}\label{cor:grid dist} For every $m,n\in \N$ and $q\in [2,\infty)$ we have
\begin{equation}\label{eq:identity forget}
\cc_q\big([m]_\infty^n\big)\asymp \min\Big\{ \dist\big(\Id_{[m]_\infty^n\to \ell_q}\big), \dist\big(\mathsf{Forget}_{[m]_\infty^n\to \ell_q}\big)\Big\}=\min\Big\{n^{\frac{1}{q}},m\Big\}.
\end{equation}
\end{corollary}

\begin{proof} It will be notationally convenient  to  show that $\cc_q([m+1]_\infty^{2n})\gtrsim \min\{n^{1/q},m\}$, which is of course equivalent to the assertion of Corollary~\ref{cor:grid dist}, up to a possible modification of the implicit universal constant factor in~\eqref{eq:identity forget}. So, suppose that $f:[m+1]_\infty^{2n}\to \ell_q$ satisfies~\eqref{eq:def distortion} with $M=[m+1]_\infty^{2n}$ and $X=\ell_q$. The task  is to deduce that $D\gtrsim \min\{n^{1/q},m\}$.

 Let $d_{\Z_{2m}^n}:\Z_{2m}^n\times \Z_{2m}^n\to \N\cup\{0\}$ be the shortest-path metric on the Cayley graph of $\Z_{2m}^n$ induced by  $\{-1,0,1\}^n\subset \Z_{2m}^n$. By~\cite[Lemma~6.12]{MN08}, there is an embedding $\psi:\Z_{2m}^n\to \{1,\ldots,m+1\}^{2n}$  whose distortion as a mapping from $(\Z_{2m}^n,d_{\Z_{2m}^n})$ to $[m+1]_\infty^{2n}$ is $O(1)$. So, there are $\alpha,\beta>0$ with $\beta/\alpha\lesssim 1$ such that $\alpha d_{\Z_{2m}^n}(x,y)\le\|\psi(x)-\psi(y)\|_\infty\le \beta d_{\Z_{2m}^n}(x,y)$ for every $x,y\in \Z_{2m}^n$. Consider the function $h=f\circ \psi:\Z_{2m}^n\to \ell_q$. Then  $\|h(x+me_i)-h(x)\|_q\ge \alpha d_{\Z_{2m}^n}(x+me_i,x)/D =\alpha m/D$ and $\|h(x+\e)-h(x)\|_q\le \beta d_{\Z_{2m}^n}(x+\e,x)\le\beta$ for all $x\in \Z_{2m}^n$, $i\in \n$ and $\e\in \{-1,1\}^n$. Therefore,  recalling that $\mathscr{K}_q(\ell_q)=1$, an application of Corollary~\ref{cor:pconvex} to $h$ gives the estimate $\alpha n^{1/q}m/D\lesssim \beta(n^{1/q}+m)$. Since $\beta/\alpha\lesssim 1$, this means that
\begin{equation*}
D\gtrsim \frac{n^{\frac{1}{q}}m}{n^{\frac{1}{q}}+m}\asymp \min\Big\{n^{\frac{1}{q}},m\Big\}.\qedhere
\end{equation*}
 \end{proof}

In the setting of Corollary~\ref{cor:grid dist}, the previous lower bound on $\cc_q([m]_\infty^n)$ is due to~\cite{MN08}, and  is smaller than the sharp estimate~\eqref{eq:identity forget} by a factor of $\sqrt{q}$. We will next see applications for which large values of $q$ are needed and using the  bound of~\cite{MN08} leads to asymptotically suboptimal results. Note that one could analogously fix $p\in (q,\infty)$ and investigate the asymptotic behavior of $\cc_q([m]_p^n)$, where $[m]_p^n$ denotes the  grid $\{1,\ldots,m\}^n\subset \R^n$ equipped with the $\ell_p$ metric on $\R^n$. In this more general setting, using Corollary~\ref{cor:grid dist} we again obtain a lower bound that is better than that of~\cite{MN08} by a factor of $\sqrt{q}$, however we no longer know that it is sharp up to universal constant factors as in~\eqref{eq:identity forget} (the best-known upper bound on $\cc_q([m]_p^n)$ follows from an argument of~\cite{MN06}). This is due to a subtlety that relates to a longstanding open question in metric embeddings; see the discussion of this intriguing issue in~\cite[Remark~1.13]{NS16} and~\cite{EN18}.

By a (special case of a) theorem of Ribe~\cite{Rib76}, for every $n\in \N$ there is $m\in \N$ such that if $\cc_p([m]_\infty^n)= O(1)$ for some $p\in [2,\infty)$, then necessarily $p\gtrsim \log n$.\footnote{A shorter proof of this fact, using  an ultrapower and differentiation argument, follows form~\cite{HM82}. Because for this particular case of Ribe's theorem the target space is $\ell_q$, a further simplification of the differentiation step is possible; see also~\cite{BL00,LN13}}  The fact that Ribe's theorem does not provide any estimate on $m$ was addressed  by Bourgain~\cite{Bou87}, who found a different proof which yields an explicit estimate (Bourgain's discretization theorem), albeit it gives here the weak bound $m=\exp(\exp(O(n)))$. However, since we are dealing with embeddings of $\ell_\infty^n$ rather than a general $n$-dimensional normed space, in the same paper he showed that it suffices to take $m=n$ in our setting (specifically, this follows from of~\cite[Proposition~5]{Bou87}); note that since $[m]_\infty^n$ contains an isometric copy of $[m]_\infty^m$ for every $m\in \n$, this implies formally that one can also take $m$ to be at most any fixed positive power of $n$, by adjusting the implicit constant in the conclusion $p\gtrsim \log n$.  Different proofs of this fact were found in~\cite{MN08,GNS12}, but it remained the best-known bound on $m$ to date. Corollary~\ref{cor:grid dist} implies that actually it suffices to take  $m=O(1)$.

To justify the latter statement, observe that if $\cc_p([m]_\infty^n)= O(1)$, then by~\eqref{eq:identity forget} we have $\min\{n^{1/p},m\}=O(1)$. Hence, provided that $m$ is bigger than an appropriate universal constant (specifically, it needs to be bigger than the implicit constant in the  $O(1)$ notation), it follows that $n^{1/p}=O(1)$, thus implying the desired lower bound $p\gtrsim \log n$. One should note the importance of having the sharp bound~\eqref{eq:identity forget} at our disposal here, because  the aforementioned weaker bound $\cc_p([m]_\infty^n)\gtrsim \min \{n^{1/p}/\sqrt{p},m\}$ of~\cite{MN08} does not imply the  lower bound $p\gtrsim \log n$ for any $m\in \N$ whatsoever.

Another perspective on the above reasoning arises by examining a parameter $p(X)\in [2,\infty)$ that was defined for a finite metric space $(X,d_X)$ in~\cite[Section~1.1]{Nao14} as the infimum over those $p\in [2,\infty)$ for which $\cc_p(X)< 10$. The  value $10$ in this definition was chosen arbitrarily in~\cite{Nao14} for notational simplicity, and clearly given $\alpha>1$ one could consider an analogous quantity $p_\alpha(X)$ by defining it to be the infimum over those $p\in [2,\infty)$ for which $\cc_p(X)<\alpha$. In addition to its intrinsic interest, the study of the quantity $p(X)$ is motivated by its algorithmic significance to approximate nearest neighbor search;  see~\cite[Remark~4.12]{Nao14} as well as~\cite{NR06,BG18} and a more recent improvement in~\cite{ANNRW18}. Due to Corollary~\ref{cor:grid dist}, we can now evaluate these parameters up to universal constant factors for $X=[m]_\infty^n$.

\begin{corollary}\label{coro:Lp index} There exists $m_0\in \N$ such that $p([m]_\infty^n)\asymp \log n$ for every $m,n\in \N$ with $m\ge m_0$. More generally, there exists a universal constant $C\ge 1$ such that $p_\alpha([m]_\infty^n)\asymp (\log n)/\log \alpha$ for every $\alpha\ge 2$ and $m,n\in \N$ with $m\ge C\alpha$.
\end{corollary}

\begin{proof} Fix $p,\alpha\ge 2$ and $m,n\in \N$. Since $\cc_p([m]_\infty^n)\le \dist(\Id_{[m]_\infty^n\to \ell_p})=n^{1/p}$, we trivially have $p_\alpha([m]_\infty^n)\le (\log n)/\log \alpha$. In the reverse direction, suppose that $\cc_p(X)<\alpha$. Therefore  $\min\{n^{1/p},m\}\lesssim \alpha$, by Corollary~\ref{cor:grid dist}. Provided that $m\ge C\alpha$ for a sufficiently large universal constant $C\ge 1$, it follows from this that $n^{1/p}\lesssim \alpha$, and hence $p\gtrsim (\log n)/\log \alpha$.
\end{proof}
Using the bound of~\cite{MN08} in place of Corollary~\ref{cor:grid dist},  one gets  $p([m]_\infty^n)\gtrsim (\log n)/\log\log n$ for $m\ge m_0$. So, Corollary~\ref{coro:Lp index} improves over the available bound  by the modest term $\log\log n$, but the resulting estimate is now sharp up to a universal constant factor. Having established this fact, one could hope for the following even more precise  understanding. Fixing $\alpha\in \{2,3,\ldots\}$ (e.g. focus here on~$\alpha=10$ as in~\cite{Nao14}), since $[m]_\infty^n$ has distortion $m$ from an $m^n$-simplex, we have $p_\alpha([m]_\infty^n)=2$ if $m<\alpha$. At the same time, Corollary~\ref{coro:Lp index} asserts that $p_\alpha([m]_\infty^n)\asymp_\alpha \log n$ if $m\ge C\alpha$, so the asymptotic behavior of $p_\alpha([m]_\infty^n)$ changes markedly as $m$ ranges over the interval $ [\alpha-1,C\alpha]$. This might occur abruptly rather than obeying intermediate asymptotics within the bounded interval; the following natural question seems accessible.

\begin{question} Fix $\alpha\in \{2,3,\ldots\}$. Do there exist $m_0=m_0(\alpha)\in \N$ and $n_0=n_0(\alpha)\in \N$ such that for every $n\in \{n_0,n_0+1,\ldots\}$ we have  $p_\alpha([m]_\infty^n)=2$ if $m\in \{1,\ldots,m_0-1\}$, yet $p_\alpha([m]_\infty^n)\gtrsim_\alpha \log n$ for $m\in \{m_0,m_0+1,\ldots\}$? If so, perhaps $m_0=\alpha$?
\end{question}

\subsection{Quadratic metric inequalities}\label{sec:hierarchy}
Here we will briefly explain how the present work answers further questions  that were posed in the  article~\cite{ANN16} of Andoni, Naor and Neiman, which investigates  issues related to the fundamental (still wide open) problem of obtaining an intrinsic characterization of those metric spaces that admit a bi-Lipschitz embedding into some Alexandrov space of nonpositive or nonnegative curvature.
This topic falls under the intriguing general question that Gromov calls the "curvature problem"~\cite[Section~$1.19_{+}$]{Gro07} (see also his earlier formulation\footnote{Specifically, in~\cite[\S15(b)]{Gro01} Gromov wrote "The geodesic property is one logical level up from concentration inequalities as it involves the {\em existential
quantifier}.
It is unclear if there is a simple $\exists$-free description of (nongeodesic!) {\em subspaces} in $\mathrm{CAT}(\kappa)$-spaces." The term "concentration inequalities" is defined in~\cite[\S15(a)]{Gro01} to be the same inequalities as the quadratic metric inequalities that we consider in~\eqref{eq:quadratic}, except that in~\cite{Gro01} they are allowed to involve arbitrary powers of the pairwise distances. However, due to~\cite{ANN16} it suffices to consider only quadratic inequalities for the purpose of the simple intrinsic description that Gromov hopes to obtain (though, as he indicates, it may not exist).} in~\cite[\S15]{Gro01}, the overall discussion of this topic in~\cite{ANN16,Bac18}, and~\cite{Stu99,FLS07,BN08,Sat09,LP10,BN18} for progress on the isometric setting).

Following~\cite{ANN16}, given $n\in \N$ and two $n$-by-$n$ matrices with nonnegative entries $\mathsf{A}=(a_{ij}),\mathsf{B}=(b_{ij})\in \mathsf{M}_n\big([0,\infty)\big)$, a metric space $(X,d_X)$ is said to satisfy the $(\mathsf{A},\mathsf{B})$-quadratic metric inequality if
\begin{equation}\label{eq:quadratic}
\forall\, x_1,\ldots,x_n\in X,\qquad \sum_{i=1}^n\sum_{j=1}^n a_{ij}d_X(x_i,x_j)^2\le \sum_{i=1}^n\sum_{j=1}^n b_{ij}d_X(x_i,x_j)^2.
\end{equation}

As explained in~\cite[Section~4]{ANN16}, there exist two collections of pairs of matrices with nonnegative entries
$$
\mathscr{A}^{\le 0},\mathscr{A}^{\ge 0} \subset \bigcup_{n=1}^\infty \Big(\mathsf{M}_n\big([0,\infty)\big)\times \mathsf{M}_n\big([0,\infty)\big)\Big),
$$
such that bi-Lipschitz embeddability into some Alexandrov space of nonpositive or nonnegative curvature is characterized by the quadratic metric inequalities that are associated to $\mathscr{A}^{\le 0}$ and $\mathscr{A}^{\ge 0}$, respectively.  Namely, a metric space $(Y,d_Y)$ admits an Alexandrov space of nonpositive (respectively, nonnegative) curvature $(X,d_X)=(X(Y),d_{X(Y)})$ for which  $\cc_X(Y)=O(1)$ if and only if $(Y,d_Y)$ satisfies the $(\mathsf{A},\mathsf{B})$-quadratic metric inequality for every $(\mathsf{A},\mathsf{B})\in \mathscr{A}^{\le 0}$ (respectively, for every $(\mathsf{A},\mathsf{B})\in \mathscr{A}^{\ge 0}$). See~\cite[Proposition~3]{ANN16} for a more refined formulation of this fact that spells out the dependence on the implicit constant in the  $O(1)$ notation. One possible (and desirable) way to obtain an intrinsic characterization of those metric spaces that admit a bi-Lipschitz embedding into some Alexandrov space of nonpositive or nonnegative curvature would be to specify concrete
families $\mathscr{A}^{\le 0},\mathscr{A}^{\ge 0}$ as above. Note that one could describe such $\mathscr{A}^{\le 0},\mathscr{A}^{\ge 0}$ by considering {\em all the possible} quadratic metric inequalities that {\em every} such Alexandrov spaces satisfies, so this question necessarily has some vagueness built into it, depending on what is considered to be "concrete" here.

A concrete candidate for $\mathscr{A}^{\ge 0}$ has not yet been proposed, and it would be very interesting to investigate this further. In~\cite[Section~5]{ANN16} a (quite complicated) candidate for $\mathscr{A}^{\le 0}$ was derived, and the question whether or not it satisfies the desired property remains an interesting open problem that is perhaps tractable using currently available methods.

It was shown in~\cite{ANN16} that whatever $\mathscr{A}^{\ge 0}$ may be, the corresponding family of inequalities "trivializes" if the distances are not squared, i.e., if $\big(\mathsf{A}=(a_{ij}),\mathsf{B}=(b_{ij})\big)\in \big(\mathsf{M}_n(\R)\times \mathsf{M}_n(\R)\big)\cap \mathscr{A}^{\ge 0}$, then for {\em every} metric space $(M,d_M)$ we have
$$
\forall\, x_1,\ldots,x_n\in M,\qquad \sum_{i=1}^n\sum_{j=1}^n a_{ij}d_M(x_i,x_j)\lesssim\sum_{i=1}^n\sum_{j=1}^n b_{ij}d_M(x_i,x_j).
$$
In light of this result, \cite[Section~1.4.1]{ANN16} naturally raised the question whether the same "trivialization" property holds for $\mathscr{A}^{\le 0}$.  Theorem~\ref{thm:main} resolves this question, as exhibited by the sharp metric cotype $2$ inequality itself, which is a quadratic metric inequality that we now know holds in any Alexandrov space of nonpositive curvature, but if one raises all the distances that occur in it to power $1$ rather than squaring them, then it is straightforward to verify that the resulting distance  inequality fails  for the metric space $(\Z_{2m}^n,d_{\Z_{2m}^n})$, even allowing for any loss of a constant factor.

An inspection of the proof of Theorem~\ref{thm:main} that appears below reveals that the sharp metric cotype $2$ inequality for Alexandrov spaces of nonpositive curvature belongs to the family of quadratic metric inequalities that were derived in~\cite[Section~5.2]{ANN16}. Checking  this assertion is somewhat tedious but entirely  mechanical. Specifically, our proof of Theorem~\ref{thm:main}  consists only of (many) applications of  the triangle inequality and (many) applications of an appropriate variant of Pisier's martingale inequality for barycentric metric spaces (see Proposition~\ref{prop:pisier} below). The latter inequality is due to~\cite{MN13}, where its proof is an iterative application of the barycentric condition~\eqref{eq:pbarycentric}. As such, the ensuing derivation of the sharp metric cotype $2$ inequality for  Alexandrov spaces of nonpositive curvature can be recast as falling into the hierarchical framework of~\cite[Section~5.2]{ANN16}. Hence, the above lack of "trivialization" of $\mathscr{A}^{\le 0}$ when the squares are removed already occurs within this hierarchy, thus answering another question that was raised in~\cite{ANN16}.

\subsubsection{Wasserstein spaces}
Given $p\ge 1$ and a  separable complete metric space $(X,d_X)$, let $\mathscr{P}_{\!p}(X)$ denote the space of all Borel probability measures on $X$ of finite $p$'th moment, equipped with the Wasserstein-$p$ metric; see e.g.~\cite{Vil09} for all of the relevant terminology and background (which will not be used in any of the ensuing proofs).

It was proved in~\cite{ANN16} that if $p>1$ and $\theta\in (0,1/p]$, then for every finite metric space $(M,d_M)$, its $\theta$-snowflake, i.e., the metric space $(M,d_M^{\theta})$, embeds with distortion arbitrarily close to $1$ into $\mathscr{P}_{\!p}(\R^3)$. Hence, by basic facts (see e.g.~\cite[Section~2.4]{KL97}) about ultrapowers of metric spaces,  $(Y,d_Y^{\theta})$ embeds isometrically into an ultrapower of $\mathscr{P}_{\!p}(\R^3)$ for every (not necessarily finite) metric space $(Y,d_Y)$. In particular, since by~\cite{Ott01,Stu06,LV09} an ultrapower of $\mathscr{P}_{2}(\R^3)$ is an Alexandrov space of nonnegative curvature, it follows that any metric space embeds via an embedding which is simultaneously coarse, uniform and quasisymmetric into some Alexandrov space of nonnegative curvature.

Initially, namely prior to the present work, the validity of the above universality result led some researchers to suspect that the conclusion of Theorem~\ref{thm:cat0 special case} might actually be false. Indeed, a main open question of~\cite{ANN16} was whether for every metric space $(Y,d_Y)$ the metric space  $(Y,\sqrt{d_Y})$ admits a bi-Lipschitz embedding into some Alexandrov space of nonpositive curvature; Theorem~\ref{thm:cat0 special case} resolves this question.

Since the proof of Theorem~\ref{thm:main bary} relies only on the sharp metric cotype of the target space, which involves only its finite subsets, the following  corollary is a  consequence of a combination of the aforementioned result of~\cite{ANN16} and Theorem~\ref{thm:main bary}.

\begin{corollary}\label{cor:wasserstein} If $p>1$, then $\mathscr{P}_{\!p}(\R^3)$ does not admit a coarse, uniform or quasisymmetric embedding into any barycentric metric space, and  hence, in particular, it admits no such embedding into any Alexandrov space of nonpositive curvature.
\end{corollary}

Because $\mathscr{P}_2(\R^3)$ is an Alexandrov space of nonnegative curvature, the following corollary is nothing more than a special case of Corollary~\ref{cor:wasserstein}, but  it seems  worthwhile to state separately.

\begin{corollary}
There exists an Alexandrov space of nonnegative curvature that does not embed coarsely into any Alexandrov space of nonpositive curvature.
\end{corollary}

 The potential validity of Corollary~\ref{cor:wasserstein} (and the underlying universality phenomenon  that was used to deduce it) for $p=1$ remains a very interesting open question; see~\cite{ANN16} for a thorough discussion of this matter, which is closely related to an old open question of Bourgain~\cite{Bou86}. Also,  the conclusion of Corollary~\ref{cor:wasserstein} with $\mathscr{P}_{\!p}(\R^3)$ replaced by $\mathscr{P}_{\!p}(\R^2)$ remains intriguingly open; see~\cite{AN17} for partial information in this direction. In relation to the latter question, observe that due to Proposition~\ref{prop:CUQ}, a positive answer to the following question  would imply that the snowflake universality phenomenon that was established in~\cite{ANN16} for $\mathscr{P}_{2}(\R^3)$ does not hold true for $\mathscr{P}_{2}(\R^2)$.

 \begin{question}\label{Q:sharp cotype wasserstein}
 Does $\mathscr{P}_{2}(\R^2)$ have sharp metric cotype $q$ for any $q\in [2,\infty)$?
 \end{question}


\begin{remark}\label{rem:r3}
Answering a question that was posed in~\cite{ANN16}, our forthcoming work~\cite{EMN18} establishes that any Alexandrov space of nonnegative curvature has metric cotype $2$. In particular, both $\mathscr{P}_{2}(\R^2)$ and  $\mathscr{P}_{2}(\R^3)$ have metric cotype $2$, but by combining the aforementioned snowflake universality of $\mathscr{P}_{2}(\R^3)$ that was proved in~\cite{ANN16} with Proposition~\ref{prop:CUQ},
we see that $\mathscr{P}_{2}(\R^3)$ fails to have metric cotype $q$ with sharp scaling parameter  for any $q\in [2,\infty)$. This yields another example of a metric space whose metric cotype is not sharp, which is yet another indication that the phenomenon of Theorem~\ref{thm:main}  is a rare occurrence. This also shows that the answer to Question~\ref{Q:sharp cotype wasserstein} with $\mathscr{P}_{2}(\R^2)$ replaced by  $\mathscr{P}_{2}(\R^3)$ is negative.
\end{remark}

The following question on understanding the nonnegative-curvature counterpart of Theorem~\ref{thm:main}  seems accessible.


\begin{question}\label{Q:cotype scaling nonnegative}
For $n\in \N$, what is the asymptotic behavior of the smallest $m(n)\in \N$ such that any Alexandrov space of nonnegative curvature has metric cotype $2$ with scaling parameter $m=m(n)$, namely~\eqref{eq:cotype} holds for $q=2$ and $\Gamma=O(1)$?
\end{question}

As stated in Remark~\ref{rem:r3}, the fact that in Question~\ref{Q:cotype scaling nonnegative} the desired $m(n)\in \N$ exists is due to our forthcoming work~\cite{EMN18}, but the bound that follows from the proof of~\cite{EMN18} is probably far from being asymptotically sharp.

\subsection{Implications of the sign of curvature under large deformations}\label{sec:table}
The present work shows that any Alexandrov space of nonpositive curvature has metric cotype $2$, and~\cite{EMN18} shows that also any Alexandrov space of nonnegative curvature has metric cotype $2$. So, on a qualitative level both "signs of curvature" imply the same (best possible, because of the requirement that the space is geodesic) metric cotype. In terms of this specific quadratic metric inequality (recall~\eqref{eq:quadratic} for the general form of a quadratic metric inequality), the difference between the two possible signs of curvature turns out to be more nuanced, namely it manifests itself in the asymptotic behavior of the scaling parameter $m$  in~\eqref{eq:cotype} (with $q=2$) as $n\to \infty$. Specifically, Theorem~\ref{thm:main q-bary} shows that $m\le\sqrt{n}$ suffices when $(X,d_X)$ is an Alexandrov space of nonpositive curvature, while by~\cite{ANN16} we know that $m/\sqrt{n}$ must tend to $\infty$ as $n\to \infty$ if $(X,d_X)$ is the Alexandrov space of nonnegative curvature $\mathscr{P}_{2}(\R^3)$ (recall Question~\ref{Q:cotype scaling nonnegative}).\footnote{This demonstrates that the role of the scaling parameter in the definition of metric cotype is not as subsidiary as it may seem from~\cite{MN08}, where it did not have a crucial role in the metric characterization of Rademacher cotype or the nonlinear Maurey--Pisier theorem. While sharp metric cotype was shown in~\cite{MN08} to have implications to coarse and uniform embeddings, there is a definite possibility that metric cotype $q$ and sharp metric cotype $q$ coincide for Banach spaces (though this is a major  open problem); here we see that this is markedly not so in the setting of Alexandrov geometry, leading to formidable qualitative differences between the coarse implications of the sign of curvature.} If one puts aside this deeper (and important) subtlety,    Alexandrov spaces of nonnegative curvature are actually "better behaved" than Alexandrov spaces of nonpositive curvature in terms of those quadratic invariants that have been computed thus far; see Table~\ref{table:sign}.

\begin{table}[h]
  \caption{Influence of the sign of curvature of Alexandrov spaces on  quadratic bi-Lipschitz invariants. }\label{table:sign}
\begin{center}
  \begin{tabular}{ | c || c | c | c |c|c|c|}
    \hline
    sign of  curvature & Enflo type 2 & Markov type 2 & metric cotype 2 & Markov convexity 2 & nonlinear spectral gap \\ \hline
    $\ge 0$ & yes & yes & yes &yes &no \\ \hline
    $\le 0$ & yes & no & yes & no & no \\
    \hline
  \end{tabular}
\end{center}
\end{table}

In addition to metric cotype, the Ribe program produced several quadratic metric invariants that have a variety of applications (at the same time, questions on how to formulate certain other invariants that should capture in a metrical way key Banach space  properties remain important "missing steps" of the Ribe program; see e.g.~\cite{Nao12}). Those include Enflo type~\cite{Enf69,Enf78} (which, in its quadratic special case, coincides with invariants that were considered by Gromov~\cite{Gro83-filling} and Bourgain, Milman and Wolfson~\cite{BMW86}), K.~Ball's Markov type~\cite{Bal92}, and Markov convexity~\cite{LNP09}. Table~\ref{table:sign} explains how the sign of the curvature of an Alexandrov space influences the validity of such invariants; e.g.~its second column indicates that Enflo type $2$ is implied by both nonnegative and nonpositive curvature (this follows by a variant of an inductive argument that essentially goes back to Enflo; see e.g.~\cite{Oht09}), whereas its third column says that, due to~\cite{Oht09}, any Alexandrov space of nonnegative curvature has Markov type $2$, while, due to~\cite{Kon12}, there exist Alexandrov spaces of nonpositive curvature that do not have Markov type $2$ (even any nontrivial Markov type). The fourth column of Table~\ref{table:sign} was already discussed above, and the top entry of its fifth column is due to~\cite{AN17} (see also~\cite{Li16,EMN18} for variants of the argument in~\cite{AN17}), while the bottom entry of the fifth column is an immediate consequence of the fact that (the $1$-dimensional simplicial complex of) a tree is an Alexandrov spaces of nonpositive curvature.

The last column of Table~\ref{table:sign} deals with the validity of the nonlinear spectral gap inequality~\eqref{def:metric spectral gap} below for {\em every} classical expander $\{\G(n)\}_{n=1}^\infty$. While, as we discussed earlier, this phenomenon holds for a variety of spaces, Table~\ref{table:sign} indicates that it fails for Alexandrov spaces of either nonpositive or nonnegative curvature, due to~\cite{Kon12} and~\cite{ANN16}, respectively.\footnote{In terms of bi-Lipschitz distortion of arbitrary expanders, Table~\ref{table:sign} does not fully indicate the extent to which nonnegative curvature behaves better than nonpositive curvature.  The fact~\cite{Oht09} that an Alexandrov space of nonnegative curvature has Markov type $2$ implies (by examining the standard random walk on the graph) that if $\mathcal{G}=\{\G(n)\}_{n=1}^\infty$ is an expander, then $\cc_X(\G(n))\gtrsim_{\mathcal{G}} \sqrt{\log |\G(n)|}$ for any nonnegatively curved Alexandrov space $X$. In contrast, there is~\cite{Kon12} an Alexandrov space of nonpositive curvature which contains an expander with $O(1)$ distortion.}

Table~\ref{table:sign} is not exhaustive. For example, our forthcoming work~\cite{EMN18} formulates a new metric invariant (inspired by~\cite{LN04}) called {\em diamond convexity}, and proves that Alexandrov spaces of either nonpositive or nonnegative curvature have diamond convexity $2$. Other metric invariants that arose from the Ribe program~\cite{NS16-KS,NS16,NS18} have not yet been computed for Alexandrov spaces, and this is so also for  invariants that arose in other (related) contexts~\cite{KN06,ANT13,LN14,NY18}.

While the literature contains several  geometric inequalities that distinguish between the sign of curvature of Alexandrov spaces (quadrilateral inequality~\cite{BN08}, tripod inequality~\cite{LP10}, Ptolemy inequality~\cite{FLS07} etc.), these had initially been relevant only for the isometric (or almost-isometric) regime, in the sense that if one considers the same inequalities up to a sufficiently large universal constant factor (for the above examples, factor $4$ suffices), then {\em any} metric space satisfies the resulting invariant. So, these initial investigations did not rule out embeddings of very large distortion into Alexandrov spaces. In recent years, some invariants that do retain their non-triviality after arbitrarily large deformations were found, but the crux of the above comparison is that they turned out to always behave either the same or better in the setting of nonnegative curvature than in the setting of nonpositive curvature. This is so even for metric cotype $2$, but the present work establishes that the asymptotic behavior of the scaling parameter in the cotype $2$ inequality provides an invariant with respect to which nonpositive curvature is in fact better behaved than nonnegative curvature. The search for such an  invariant was one of the  motivations of the present investigation, but there clearly remains much to be done in order to understand the coarse/large distortion implications of the sign of curvature in Alexandrov geometry (this general theme is in the spirit of the Riemannian counterpart  that is presented in~\cite{Gro91}).

\section{Limitations of previous approaches and further open questions}\label{sec:open}

The main purpose of this section is to explain why Theorem~\ref{thm:cat0 special case} does not follow from the previously available results (which we listed in the Introduction) on obstructions to coarse embeddings. The ensuing synthesis of the literature and examination of the conceptual limitations of existing tools to address Gromov's question  also naturally leads us to state  several open problems; we believe that these are rich, fertile and important directions for future research.

While the present section is helpful for understanding the significance of our results as well as the challenges that remain, it will not be used in the subsequent proofs and therefore it could be skipped on first reading.

\smallskip

We will next elaborate on all of the pointers to the literature that were listed at the start of the Introduction, but it will be  instructive to do so thematically rather than following the chronological order that we used previously.

\subsection{Metric cotype} The direct precursor (and inspiration) of the present work is~\cite{MN08}, where the validity of sharp metric cotype $q$ is demonstrated for certain Banach spaces. The proof of~\cite{MN08} uses the underlying linear structure through an appeal to the boundedness of an oscillatory convolution operator (the Rademacher projection), which relies on cancellations.  We do not see how to interpret that strategy of~\cite{MN08} for targets that are not Banach spaces.

\subsection{Negative definite kernels} The works of Johnson and Randrianarivony~\cite{JR06} and Randrianarivony~\cite{Ran06} proceed via a reduction to the linear theory through an influential approach of Aharoni, Maurey and Mitaygin~\cite{AMM85}, which uses negative definite kernels and was invented to treat uniform embeddings. This strategy relies heavily on the underlying linear (even Hilbertian) structure and we doubt that it could be interpreted  for targets that are not Banach spaces. If such an interpretation were possible, then it would be a very interesting achievement, likely of value elsewhere.

\subsection{The Urysohn space} Parts of Pestov's proof~\cite{Pes08} that the Urysohn space $\mathbb{U}$ does not admit a coarse embedding into any uniformly convex Banach space $(X,\|\cdot\|_X)$ rely on linear considerations (e.g.~using the Hahn--Banach theorem), so at the outset its relevance to the fully nonlinear setting of Theorem~\ref{thm:cat0 special case} is questionable. But, even putting this matter aside, there is the following  conceptual reason why the approach of~\cite{Pes08} is ill-suited to treating nonpositively curved targets. The crux of the argument in~\cite{Pes08} is the recursive use of   a strong self-similarity property~\cite{Hru92,Sol05,Ver08} of  $\mathbb{U}$ to demonstrate that if $\mathbb{U}$  embeds coarsely into  $X$, then an ultrapower of $\ell_2(X)$ contains arbitrarily large complete binary trees in which all pairs of vertices that are either siblings or form an edge are distorted by at most a $O(1)$-factor (see~\cite[Section~6]{Pes08} for a precise formulation of this statement). By a classical result in Banach space theory (see e.g.~\cite[Chapter~9]{FHPMPZ01}), this conclusion contradicts  the premise that $X$ is uniformly convex. However, (the $1$-dimensional simplicial complex of) a connected tree is  an Alexandrov space of nonpositive curvature, so the above approach  detects a geometric structure that is compatible with nonpositive curvature. Note that, in hindsight, due to universality properties of $\mathbb{U}$,  Theorem~\ref{thm:cat0 special case}  implies formally that $\mathbb{U}$ does not embed coarsely into any Alexandrov space of nonpositive curvature, but the reason that we obtain here for this fact is entirely different from the strategy of~\cite{Pes08}.

\subsection{Arbitrary expanders}\label{sec:any expander}

Throughout  what follows, graphs are tacitly assumed to be finite, connected and regular.  The vertex set of a graph $\G$ is denoted  $V_\G$ and its edge set is denoted $E_\G$. The shorted-path metric that $\G$  induces  on $V_\G$  is denoted $d_{\G}$. When we consider a graph $\G$ as a metric space, it is always understood to be $(V_\G,d_\G)$.

A sequence of graphs $\{\G(n)\}_{n=1}^\infty$ is an expander with respect to a metric space $(X,d_X)$ if $\lim_{n\to \infty }|V_{\G(n)}|=\infty$, the degree of $\G(n)$ is bounded above independently of $n\in \N$, and there exists $\gamma\in (0,\infty)$ such that for any $n\in \N$ we have
\begin{equation}\label{def:metric spectral gap}
\forall \{x_u\}_{u\in V_{\G(n)}}\subset X,\qquad \frac{1}{|V_{\G(n)}|^2}\sum_{(u,v)\in V_{\G(n)}\times V_{\G(n)}} d_X(x_u,x_v)^2\le \frac{\gamma}{|E_{\G(n)}|} \sum_{\{u,v\}\in E_{\G(n)}} d_X(x_u,x_v)^2.
\end{equation}

An extensive  discussion of  expanders with respect to metric spaces can be found in~\cite{MN-SODA,MN14,MN15}. In the special case when $\{\G(n)\}_{n=1}^\infty$  is an expander with respect to a Hilbert space (equivalently with respect to $\R$, or even, by an application of Cheeger's inequality~\cite{Che70,Fie73,Dod84,AM85}, with respect to $\{0,1\}$), the common simpler terminology is that $\{\G(n)\}_{n=1}^\infty$  is an expander (in the classical sense); detailed background on this important notion appears in e.g.~\cite{HLW06}.

 As noted in~\cite{Gro00}, by a short  argument (that is also implicit in works of Linial, London and Rabinovich~\cite{LLR95} and Matou\v{s}ek~\cite{Mat97}, both of which considered bi-Lipschitz embeddings), one shows that if $\{\G(n)\}_{n=1}^\infty$ is  an expander with respect to $(X,d_X)$, then any metric space $(Y,d_Y)$ that contains an isometric copy of each of the finite metric spaces $\{(V(\G_n),d_{\G(n)})\}_{n=1}^\infty$ (e.g.~their disjoint union or their Pythagorean product) does not embed coarsely into $(X,d_X)$. In particular, since (classical) expanders are known to exist (by now, via a wealth of different constructions; see~\cite{HLW06}), we thus obtain Gromov's proof~\cite{Gro00,Gro03} that there exists a metric space that does not embed coarsely into a Hilbert space.

As indicated by Gromov~\cite{Gro01,Gro03} and proven in detail by Kondo~\cite{Kon12}, there exists an Alexandrov space of nonpositive curvature that contains with $O(1)$ distortion a classical expander; for an even more "wild" example of such an Alexandrov space (with respect to which random regular graphs are almost surely {\em not} expanders), see~\cite{MN15}. In particular, it follows that some Alexandrov spaces of nonpositive curvature do not admit a coarse embedding into a Hilbert space. It also follows that the aforementioned important realization of~\cite{Gro00,Gro03} that the presence of {\em any expander whatsoever} implies coarse non-embeddability into a Hilbert space is irrelevant for proving Theorem~\ref{thm:main}.

\subsection{Arbitrary relative expanders} The notion of a {\em relative expander}, or an expander relative to a partition,  was introduced  by  Arzhantseva and Tessera~\cite{AT15} for studying aspects of coarse embeddings into a Hilbert space; it was also used implicitly in the earlier work~\cite{ALW01}, and was studied for algorithmic purposes (clustering) in~\cite{ACKS16}. We will now present the obvious generalization  of the definition of~\cite{AT15}  to a notion of a relative expander with respect to a metric space.

Say that a sequence of graphs $\{\G(n)\}_{n=1}^\infty$ is a relative expander with respect to a metric space $(X,d_X)$ if the degree of $\G(n)$ is bounded above independently of $n\in \N$, and  there exists a constant $\rho>0$ such that  for every $n\in \N$ there is a partition $\mathcal{P}(n)=\{C_1(n),C_2(n),\ldots,C_{k(n)}(n)\}$ of $V_{\G(n)}$ for which $\lim_{n\to \infty}\min_{j\in \{1,\ldots,k(n)\}}|C_j(n)|=\infty$, and
\begin{equation}\label{eq:def relative expander}
\forall \{x_u\}_{u\in V_{\G(n)}}\subset X,\qquad  \frac{1}{|V_{\G(n)}|}\sum_{j=1}^{k(n)} \frac{1}{|C_j(n)|}\sum_{u,v\in C_j(n)} d_X(x_u,x_v)^2\le \frac{\rho }{|E_{\G(n)}|}\sum_{\{u,v\}\in E_{\G(n)}} d_X(x_u,x_v)^2.
\end{equation}
When $(X,d_X)$ is a Hilbert space one simply says that $\{\G(n)\}_{n=1}^\infty$ is a relative expander. This is a weakening of the definition~\ref{def:metric spectral gap} of an expander with respect to $(X,d_X)$, which corresponds to the special case when the partition $\mathcal{P}(n)$ is the trivial partition $\{V_{\G(n)}\}$ for every $n\in \N$. It is, in fact, a very substantial weakening, because by~\cite{AT15} there is a relative expander $\{\G(n)\}_{n=1}^\infty$ such that $\bigsqcup_{n=1}^\infty \mathsf{H}(n)$ does not embed coarsely into $\bigsqcup_{n=1}^\infty \mathsf{G}(n)$ for {\em any} expander $\{\mathsf{H}(n)\}_{n=1}^\infty$. Despite this, the deduction that~\eqref{def:metric spectral gap} implies that  $\bigsqcup_{n=1}^\infty \mathsf{G}(n)$ does not embed coarsely into $(X,d_X)$ carries over effortlessly\footnote{The only additional observation  that is needed for this is that, because the degree of $\G(n)$ is $O(1)$, a quick and standard counting argument (see e.g.~the justification of equation (36) in~\cite{NR17}) shows that $d_{\G(n)}(u,v)\gtrsim \log |C_j(n)|\to \infty$ for a constant fraction of $(u,v)\in C_j(n)\times C_j(n)$.} to show that this remains valid  under the weaker hypothesis that $\{\G(n)\}_{n=1}^\infty$ is a relative expander with respect to $(X,d_X)$.

Because, as we  discussed above, one cannot use arbitrary expanders to prove Theorem~\ref{thm:cat0 special case}, also arbitrary relative expanders cannot be used for this purpose, since the latter is an even larger family of (sequences of) graphs. When in Theorem~\ref{thm:cat0 special case} the target space $X$ is a simply connected Riemannian manifold of nonpositive sectional curvature rather than a general (potentially singular) Alexandrov space of nonpositive curvature, by~\cite{Wan98,IN05} one can take $Y$ to be $\bigsqcup_{n=1}^\infty \mathsf{G}(n)$ for any expander $\{\G(n)\}_{n=1}^\infty$. So, arbitrary expanders  suffice for targets that are not singular (or have "bounded singularities"~\cite{Gro03,Toy10}), but the proof of this in~\cite{Wan98,IN05,NS11}  does not carry  over to relative expanders.

\begin{question}\label{Q:rel ex} Does there exist a relative expander $\{\G(n)\}_{n=1}^\infty$ and a simply connected Riemannian manifold $(M,d_M)$ of nonpositive sectional curvature such that $\bigsqcup_{n=1}^\infty \mathsf{G}(n)$  embeds coarsely into $(M,d_M)$?
\end{question}
Despite the fact that Question~\ref{Q:rel ex} is inherently not a route towards a different proof of Theorem~\ref{thm:cat0 special case},  it is a natural (and perhaps quite accessible) question in Riemannian geometry that arises from the present considerations.

\subsection{Specially-crafted expanders} In order to apply the idea that we recalled in Section~\ref{sec:any expander} to show that there exists a metric space that does not embed coarsely into some non-Hilbertian metric space $(X,d_X)$, one only needs to show that $(X,d_X)$ admits {\em some } expander rather than to show that  {\em any} (classical) expander whatsoever is also an expander with respect to $(X,d_X)$. See~\cite{Mat97,Oza04,IN05,Toy10,NS11,Pis10,Nao14,Mim15,Che16,Toy16,Nao17} for theorems that provide a variety of spaces $(X,d_X)$ that do satisfy the latter stronger requirement for any possible expander. In general, however, it could be that no expander with respect to $(X,d_X)$  exists (certainly not with respect to, say,  $\ell_\infty$). This is so even for (seemingly) "nice" spaces; e.g.~by~\cite{ANN16} there is an Alexandrov space of nonnegative curvature, namely $\mathscr{P}_2(\R^3)$, with respect to which no sequence of graphs is an expander. Obtaining a useful/workable intrinsic  characterization of  those metric spaces with respect to which some expander exists is an important (likely difficult, perhaps intractable) open question.

 This leaves the  possibility, as a potential alternative route towards a proof  Theorem~\ref{thm:cat0 special case}, that one could somehow  come up with a special  sequence of graphs $\{\G(n)\}_{n=1}^\infty$  which is simultaneously an expander with respect to {\em every} Alexandrov space of nonpositive curvature. If this were indeed possible, then it would require constructing a specially-crafted expander, and finding a way to prove the quadratic distance inequality~\eqref{def:metric spectral gap} that relies on geometric considerations (based solely on nonpositive curvature) rather than the straightforward linear algebra/spectral consideration that underlies its Euclidean counterpart (recall that by~\cite{ANN16} this cannot be accomplished for nonnegative curvature).

\begin{question}\label{Q:do expanders exist}
Is there a sequence of graphs $\{\G(n)\}_{n=1}^\infty$ which is an expander with respect to every Alexandrov space of nonpositive curvature $X$? More modestly, is there such an expander  that does not embed coarsely into any such $X$?
\end{question}

Question~\ref{Q:do expanders exist} does not originate in the present work, and is in fact  a well-known  open problem that has been broached  several times (see e.g.~\cite{IN05,Toy10,NS11,FT12,Kon12,MN14,AT15,MN15,ANN16,Toy16}); we restated it above due to its relevance to the present discussion (and importance). Note that if Question~\ref{Q:do expanders exist} had a positive answer, then, since expanders have bounded degree, the disjoint union of the resulting expander $\{\G(n)\}_{n=1}^\infty$  would be a space of bounded geometry  that fails to admit a coarse embedding into any Alexandrov space of nonpositive curvature;  thus answering Question~\ref{Q:bounded geometry}.

V.~Lafforgue devised~\cite{Laf08,Laf09} a brilliant method to prove that certain special graph sequences are expanders with respect to every Banach space of nontrivial Rademacher type; see also~\cite{Lia14,LS15,dlS16} for further perspectives on such {\em super-expanders}. Lafforgue's approach relies heavily on the linear structure of the underlying Banach space through the use of Fourier-analytic considerations, including appeal to  the   substantial  work~\cite{Bou82}. As such, we do not see how these methods could be relevant to Question~\ref{Q:do expanders exist}, though it would be very interesting if they could be implemented for metric spaces that are not Banach spaces. Further examples of such expanders~\cite{Saw17,dLV17,FNvL17} were more recently obtained using Roe's warped cones~\cite{Roe05} and a discretization procedure of Vigolo~\cite{Vig16}, but these take Lafforgue's expanders as input and therefore they suffer from the same deficiency vis-\`a-vis Question~\ref{Q:do expanders exist}.

An entirely different strategy to construct super-expanders was found by Mendel and Naor~\cite{MN14}, using  the zigzag graph product of Reingold, Vadhan and Wigderson~\cite{RVW02}. This strategy does apply to targets that are not Banach spaces (see e.g.~its implementation in~\cite{MN15}), but its potential applicability to Question~\ref{Q:do expanders exist} remains unclear because it relies on an iterative construction and at present we lack a "base graph" to start the induction. Due to~\cite{MN13}, finding such a base graph is the {\em only} issue that remains to be overcome in order to answer Question~\ref{Q:do expanders exist} (positively) via the zigzag strategy of~\cite{MN14,MN15}. This seems to be a difficult question; see Section~\ref{sec:quotient cube} for a candidate base graph.

A more ambitious strengthening of  Question~\ref{Q:bounded geometry} (which is also a well-known problem)  is the following question.

\begin{question}\label{eq:group}
Does there exist a finitely generated group (equipped with a word metric associated to any finite symmetric set of generators) that does not embed coarsely into any Alexandrov space of nonpositive curvature?
\end{question}

As explained in~\cite{NS11}, by applying the graphical random group construction of~\cite{Gro03} (see also~\cite{AD08,Osa14}), a positive answer to Question~\ref{Q:girth} below would imply a positive  answer to Question~\ref{eq:group}. By~\cite{NS11}, this would also yield a group $G$ such that any isometric action of $G$ on an Alexandrov space of nonpositive curvature has a fixed point; the existence of such a group is itself a major open problem, in addition to Question~\ref{eq:group} and Question~\ref{Q:girth}.

\begin{question}\label{Q:girth}
 Does there exist a sequence of graphs $\{\G(n)\}_{n=1}^\infty$ which is an expander with respect to every Alexandrov space of nonpositive curvature, and also for every $n\in \N$ the girth of $\G(n)$ is at least $c\log |V_{\G(n)}|$ for some $c\in (0,\infty)$?
\end{question}

Unfortunately, the approaches of~\cite{Laf08,Laf09,MN14} for constructing super-expanders seem inherently ill-suited for producing the high-girth graphs that Question~\ref{Q:girth} aims to find.

\subsection{Quotients of the Hamming cube}\label{sec:quotient cube} Following~\cite{MN14},
 a candidate for the  base graph in the inductive construction of~\cite{MN14} arises from the work~\cite{KN06} of Khot and Naor, where it is shown that for every $n\in \N$ there is a linear subspace  $V_n\subset \mathbb{F}_2^n=\{0,1\}^n$, namely the polar of an "asymptotically good linear code,"  such that for every $f:\mathbb{F}_2^n/V_n\to \ell_2$ satisfies
\begin{equation}\label{eq:code l2}
\frac{1}{4^n}\sum_{(x,y)\in \mathbb{F}_2^n\times \mathbb{F}_2^n}\|f(x+V_n)-f(y+V_n)\|_2^2\lesssim \frac{1}{n2^n}\sum_{j=1}^n \sum_{x\in \mathbb{F}_2^n}\|f(x+e_j+V_n)-f(x+V_n)\|_2^2.
\end{equation}
The estimate~\eqref{eq:code l2} is close to the assertion that the quotient Hamming graphs $\{\mathbb{F}_2^n/V_n\}_{n=1}^\infty$ (namely, the Cayley graphs of the Abelian groups $\mathbb{F}_2^n/V_n$ with respect to the generators $\{e_1,\ldots,e_n\}$) form an expander, except that they do not have bounded degrees. Nonetheless, it follows from~\eqref{eq:code l2},  similarly to the aforementioned application of~\eqref{def:metric spectral gap}, that any metric space that contains these quotient Hamming graphs does not embed coarsely into a Hilbert space (for this, one only needs to estimate the average distance in these graphs, as done in~\cite{KN06}). In the same vein, a positive answer to Question~\ref{Q:CAT quotient} below would yield a different proof of Theorem~\ref{thm:cat0 special case}. Much more significantly, by~\cite{MN13,MN14,MN15} this would yield the aforementioned desired base graph so as to answer Question~\ref{Q:do expanders exist} positively; see~\cite[Section~7]{MN14} for a variant of this approach using the heat semigroup on $\mathbb{F}_2^n$ which would also yield such a base graph, as well as a closely related harmonic-analytic question in~\cite[Section~5]{MN14} that remains open even for uniformly convex Banach spaces.

\begin{question}\label{Q:CAT quotient} Let $V_n\subset \mathbb{F}_2^n$ be linear subspaces (namely, polars of asymptotically good codes) as in~\cite{KN06} and above. Is it true that for every Alexandrov space of nonpositive curvature $(X,d_X)$, for every $f:\mathbb{F}_2^n/V_n\to X$ we have
\begin{equation}\label{eq:code CAT0}
\frac{1}{4^n}\sum_{(x,y)\in \mathbb{F}_2^n\times \mathbb{F}_2^n}d_X\big(f(x+V_n),f(y+V_n)\big)^2\lesssim \frac{1}{n2^n}\sum_{j=1}^n \sum_{x\in \mathbb{F}_2^n}d_X\big(f(x+e_j+V_n),f(x+V_n)\big)^2.
\end{equation}
\end{question}
Despite the fact that the proof of~\eqref{eq:code l2}  in~\cite{KN06} is Fourier-analytic (this is so also for the variants in~\cite{MN14}), we believe that establishing its analogue~\eqref{eq:code CAT0} for Alexandrov spaces of nonpositive curvature  is  currently the most viable approach towards Question~\ref{Q:do expanders exist}. Nevertheless, it seems that a substantial new idea is required here.

It should be noted that while~\cite{KN06} shows that quotient Hamming graphs $\{\mathbb{F}_2^n/V_n\}_{n=1}^\infty$ yield a solution of the Hilbertian case of Gromov's question that we quoted earlier, we do not know if this inherently does not follow from a reduction  to the case of classical expanders, as formulated in the following natural geometric question.

\begin{question}\label{Q:does quotient contain expander} Continuing with the above notation, does there exist a (bounded degree, classical) expander $\{\G(n)\}_{n=1}^\infty$ such that $\bigsqcup_{n=1}^\infty \G(n)$ embeds coarsely into $\bigsqcup_{n=1}^\infty (\mathbb{F}_2^n/V_n)$?
\end{question}

\subsection{Generalized roundness} The first proof that there exists a metric space which does not embed coarsely into a Hilbert space was obtained by  Dranishnikov, Gong, Lafforgue, and Yu~\cite{DGLY02}. They showed that this is so for coarse embeddings into any metric space that satisfies a classical condition, called generalized roundness $p\in (0,\infty)$, which was introduced and used by Enflo~\cite{Enf69-smirnov} to answer an  old question of Smirnov~\cite{Gor59} whether every metric space embeds uniformly into a Hilbert space (strictly speaking, \cite{DGLY02} considered only the case $p=2$, but the argument works mutatis mutandis for any $p>0$). A unified treatment of these approaches, so as to rule out coarse and uniform embeddings simultaneously, appears in~\cite{KMOW12}. As noted by several authors~\cite{Rob98,LP06,PW09,KMOW12}, a "vanilla" application of the approach of~\cite{DGLY02} fails to apply to targets that are  general Alexandrov spaces of nonpositive curvature, i.e., to prove Theorem~\ref{thm:cat0 special case}. Indeed, a combination of the characterization of generalized roundness in~\cite{LTW97} with the work~\cite{FH74} shows that there are Alexandrov spaces of nonpositive curvature that do not have positive generalized roundness; specifically, this is so for the quaternionic hyperbolic space (and not so for the real and complex hyperbolic spaces~\cite{Gan67,FH74}). However, the quaternionic hyperbolic space embeds coarsely into a Hilbert space (e.g.~by~\cite{Sel92}), so the conclusion of Theorem~\ref{thm:cat0 special case} does hold for it by a direct reduction to the case of Hilbertian targets. In other words, the correct setting of the method of~\cite{DGLY02} is those spaces that embed coarsely into some metric space of positive generalized roundness rather than those spaces that have positive generalized roundness themselves. A  decisive demonstration that the approach of~\cite{DGLY02} is inherently ill-suited for proving Theorem~\ref{thm:cat0 special case} is that, by combining~\cite{DGLY02} with the classical Schoenberg embedding criterion~\cite{Sch38}, one sees that if a metric space has positive generalized roundness, then it embeds coarsely into a Hilbert space. Hence the aforementioned Gromov--Kondo spaces are  Alexandrov spaces of nonpositive curvature that do not embed coarsely into any metric space of positive generalized roundness. (Incidently, this answers questions that were posed in~\cite[page~155]{LP06} and~\cite[page~10]{PW09}, but it seems to have gone  unnoticed that~\cite{Kon12} resolves them.)

\subsection{Flat tori} For every $n\in \N$ let $\langle\cdot,\cdot\rangle:\R^n\times \R^n\to \R$ be the standard scalar product on $\R^n$, which induces the Hilbert space $\ell_2^n$. For each $n\in \N$ fix a lattice $\Lambda_n\subset \R^n$ for which the dual lattice $\Lambda_n^*=\bigcap_{y\in \Lambda_n}\{x\in \R^n:\ \langle x,y\rangle\in \Z\}$ is a $O(\sqrt{n})$-net of $\ell_2^n$, namely $\|x-y\|_2\gtrsim \sqrt{n}$ for all distinct $x,y\in \Lambda_n^*$, and for every $z\in \R^n$ there exists $w\in \Lambda_n^*$ such that $\|z-w\|_2\lesssim \sqrt{n}$. The fact that such lattices exist is a classical theorem of Rogers~\cite{Rog50}; see also~\cite{But72,Bou87b} for other proofs.

The reasoning in~\cite[Section~4]{KN06} implies that if we endow the Abelian group $\R^n/\Lambda_n$ with the  (flat torus) quotient  Riemannian metric $d_{\R^n/\Lambda_n}:(\R^n/\Lambda_n)\times (\R^n/\Lambda_n)\to [0,\infty)$ that is induced by the $\ell_2^n$ metric, then the Pythagorean product
\begin{equation}\label{eq:def torus product}
\mathfrak{T}\eqdef\Bigg(\bigoplus_{n=1}^\infty \big(\R^n/\Lambda_n,d_{\R^n/\Lambda_n}\big)\Bigg)_2
\end{equation}
does not embed coarsely into $L_1(\R)$, and hence also it does not embed coarsely into a Hilbert space. This is so because by (a change of variable in)~\cite[Lemma~9]{KN06}, if we let $\mu_n$ denote the normalized Riemannian volume measure on the torus $\R^n/\Lambda_n$ and $\gamma_n$ denote the standard Gaussian measure on $\R^n$, then every measurable $\phi:\R^n/\Lambda_n\to L_1(\R)$ satisfies
\begin{equation}\label{eq:quote torus}
\iint_{(\R^n/\Lambda_n)\times (\R^n/\Lambda_n)}\|\phi(u)-\phi(v)\|_1\ud\mu_n(u)\ud\mu_n(v)\lesssim \iint_{(\R^n/\Lambda_n)\times \R^n} \Big\|\phi\Big(w+\frac{z}{\sqrt{n}}+\Lambda_n\Big)-\phi(w)\Big\|_1\ud\mu_n(w)\ud\gamma_n(z).
\end{equation}
If in addition we have $\omega(d_{\R^n/\Lambda_n}(u,v))\le \|\phi(u)-\phi(v)\|_1\le \Omega(d_{\R^n/\Lambda_n}(u,v))$ for all $u,v\in \R^n/\Lambda_n$ and some nondecreasing moduli $\omega,\Omega:[0,\infty)\to [0,\infty)$, then it follows from~\eqref{eq:quote torus} that $\omega(c\sqrt{n})\lesssim \Omega(1)$, where $c>0$ is a universal constant. For this computation, use the proof of~\cite[Lemma~10]{KN06}  to deduce that $d_{\R^n/\Lambda_n}(u,v)\gtrsim \sqrt{n}$ for a constant proportion (with respect to $\mu_n\times \mu_n$) of $(u,v)\in (\R^n/\Lambda_n)\times (\R^n/\Lambda_n)$, combined with the fact that $\|\phi(a)-\phi(b)\|_1\le \Omega(1)(d_{\R^n/\Lambda_n}(a,b)+1)$ for every $(a,b)\in (\R^n/\Lambda_n)\times (\R^n/\Lambda_n)$, which follows from a standard application of the triangle inequality.

The proof of~\eqref{eq:quote torus} in~\cite{KN06} is Fourier-analytic, and as such it does not apply when $L_1(\R)$ is replaced by a target which is not a Banach space. It would be very interesting if one could find a way to reason analogously about embeddings of $\mathfrak{T}$ into an Alexandrov space of nonpositive curvature.  Investigating the geometry of  $\mathfrak{T}$ is interesting in its own right, and some aspects of this were discussed in~\cite{Nao14,ANN16}. It follows immediately from~\eqref{eq:def torus product} that $\mathfrak{T}$ is an Alexandrov space of nonnegative curvature, and we do not know if every (classical) expander is also an expander with respect to $\mathfrak{T}$, or  even if $\mathfrak{T}$ admits some expander. Determining whether $\mathfrak{T}$ has sharp metric cotype $q$ for any $q\ge 2$ would also be worthwhile.

\subsection{Metric Kwapie\'n--Sch\"utt} Metric $\mathrm{KS}$ inequalities  are metric invariants that were introduced in~\cite{NS16-KS} (the nomenclature refers to the works~\cite{KS85,KS89} of Kwapie\'n and Sch\"utt). The reasoning of~\cite[Section~3]{NS16-KS} shows that if  a metric space satisfies such an inequality, then it does not admit a coarse embedding of $\ell_2(\ell_p)$ for any $p\in [1,2)$. Thus, $\ell_2(\ell_p)$ fails to satisfy any metric $\mathrm{KS}$ inequality despite the fact that its $2$-barycentric constant tends to $1$ as $p\to 2^-$. Also, it was shown in~\cite[Section~3]{NS16-KS} that if $p\in (2,\infty)$, then $\ell_p$ fails to satisfy   any metric $\mathrm{KS}$ inequality even though it is $p$-barycentric with constant  $1$. We do not know how to interpret for spaces that are not Banach spaces the Fourier-analytic proof of~\cite[Theorem~2.1]{NS16-KS}  that a Hilbert space satisfies the quadratic metric $\mathrm{KS}$ inequality. In particular, we do not know if  Alexandrov spaces of nonpositive curvature satisfy the quadratic metric $\mathrm{KS}$ inequality; the above examples show that in the (in our opinion unlikely) event that this were true, then its proof must use the fact that such spaces are $2$-barycentric, and moreover that their $2$-barycentric constant equals $1$, i.e., for any $\e>0$ it fails for spaces that are either $2$-barycentric with constant $1+\e$, or are $(2+\e)$-barycentric with constant $1$.

\subsection{Stable metrics and interlacing graphs} Following Garling's definition~\cite{Gar82}, which extends the fundamental contribution of Krivine and Maurey~\cite{KM81} in the setting of Banach spaces,  a metric space $(X,d_X)$ is said to be stable if
\begin{equation}\label{eq:def stable}
\lim_{m\to \infty}\lim_{n\to \infty} d_X(x_m,y_n)=\lim_{n\to \infty} \lim_{m\to \infty} d_X(x_m,y_n)
\end{equation}
for any bounded sequences $\{x_m\}_{m=1}^\infty,\{y_n\}_{n=1}^\infty\subset X$ for which both iterated limits in~\eqref{eq:def stable} exist.  It is simple to check that locally compact metric spaces and Hilbert spaces are stable; see the survey~\cite{Bau14} for more examples and non-examples.

In~\cite{Kal07}, Kalton investigated the use of stability of metrics to rule out coarse embeddings, profoundly building on (and inspired by) a classical work of Raynaud~\cite{Ray83} which studied stable metrics in the context of uniform embeddings. Among the results of~\cite{Kal07}, Kalton considered for each $r\in \N$ the following infinite graph, which we denote by $\mathsf{K}_r(\N)$. The vertices of $\mathsf{K}_r(\N)$ are the subsets of $\N$ of cardinality $r$. Two vertices $\sigma,\tau$ of $\mathsf{K}_r(\N)$ are declared to form an edge if they interlace, i.e., if one could write $\sigma=\{m_1,\ldots,m_r\}$ and $\tau=\{n_1,\ldots,n_r\}$ such that either $m_1\le n_1\le m_2\le n_2\le\ldots\le m_r\le n_r$ or $n_1\le m_1\le n_2\le m_2\le\ldots\le n_r\le m_r$. Below, $\mathsf{K}_r(\N)$  will always be understood to be equipped with the shortest-path metric that this graph structure induces. By a short iterative application of the  definition of stability of a metric space (see~\cite[Lemma~9.19]{BL00}), it was shown in~\cite{Kal07} that  $\bigsqcup_{r=1}^\infty \mathsf{K}_r(\N)$  does not embed coarsely into any stable metric space.
\begin{question}\label{Q:embed into stable}
Does every Alexandrov space of nonpositive curvature embed coarsely  into some stable metric space?
\end{question}
By the aforementioned result from~\cite{Kal07}, a positive answer to Question~\ref{Q:embed into stable} would imply that $\bigsqcup_{r=1}^\infty \mathsf{K}_r(\N)$ does not embed coarsely into any Alexandrov space of nonpositive curvature. This would yield a way to prove  Theorem~\ref{thm:cat0 special case} that is entirely different from how we proceed here, though note that if $r\in \{2,3,\ldots\}$, then $\mathsf{K}_r(\N)$ is not locally finite, and this is an inherent attribute of this approach because locally compact metric spaces are stable. Beyond its mere applicability to potentially proving Theorem~\ref{thm:cat0 special case}, a positive answer to Question~\ref{Q:embed into stable} would be important in its own right, as a nontrivial structural consequence of nonpositive curvature (which,  by~\cite{ANN16}, fails for nonnegative curvature). In our opinion, if true, then this would be a fundamental result that will likely have applications elsewhere, though we suspect that the answer to Question~\ref{Q:embed into stable} is negative, and that proving this might be quite accessible. Regardless, it would be interesting to determine if $\bigsqcup_{r=1}^\infty \mathsf{K}_r(\N)$ can embed coarsely into some Alexandrov space of nonpositive curvature.

The above idea of~\cite{Kal07} (partially building on~\cite{Ray83})  inspired a series of investigations~\cite{Kal07,KR08,Kal11,BLS18,LPP18,BLMS18} over recent years that led  to major coarse non-embeddability results for certain Banach spaces, starting with Kalton's incorporation~\cite{Kal07} of Ramsey-theoretic  reasoning which led to (among other things) his proof in~\cite{Kal07} that $c_0$ does not embed coarsely into any reflexive Banach space. We will not survey these ideas here, and only state that they rely on the linear theory in multiple ways, so their relevance to the setting of Theorem~\ref{thm:cat0 special case} is questionable.

\section{Proof of Theorem \ref{thm:main}}

Prior to carrying out the proof of Theorem~\ref{thm:main}, we will quickly present a probabilistic tool on which it relies.

\subsection{Nonlinear martingales} \label{subsec:martingales}  We will next describe  basic facts about martingales that don't necessarily take values in a Banach space. There are multiple ways to extend the linear theory of martingales, but we will only discuss one such approach, namely following that of~\cite[Section~2]{MN13},  on which we will rely in the proof of Theorem~\ref{thm:main} . There is substantial literature on martingales in metric spaces (including~\cite{Dos62,Eme89,EH99,Sturm02,CS08}), but because our sole purpose here is to use nonlinear martingales to prove a purely geometric result rather than the (independently interesting) foundational probabilistic perspective, we will not delve into the general theory and comparison of different approaches.

Let $\Omega$ be a finite set and let $\ms{F}\subseteq2^\Omega$ be a $\sigma$-algebra. For every $\omega\in \Omega$, define $\ms{F}(\omega)\subseteq\Omega$ to be the unique atom of $\ms{F}$ for which $\omega\in\ms{F}(\omega)$. Suppose that $\mu:2^\Omega\to [0,1]$ is a probability measure of full support, i.e., $\mu(\{\omega\})>0$ for every $\omega\in \Omega$. Let $X$ be a set equipped with a barycenter map $\mathfrak{B}:\ms{P}_X^{<\infty}\to X$. For a function  $Z:\Omega\to X$, define its $\mu$-conditional barycenter $\mathfrak{B}_\mu(Z|\ms{F}):\Omega\to X$ by setting
\begin{equation} \label{eq:conditionalbarycenter}
\forall\, \omega\in \Omega,\qquad \mathfrak{B}_\mu\big(Z|\ms{F}\big)(\omega) = \mathfrak{B} \Bigg(\frac{1}{\mu\big(\ms{F}(\omega)\big)} \sum_{a\in\ms{F}(\omega)} \mu(a) \delta_{Z(a)} \Bigg).
\end{equation}
Fix $n\in \N$ and  $\sigma$-algebras $\ms{F}_0\subset \ms{F}_1\subset \ldots\subset \ms{F}_{n}\subset 2^\Omega$ (a filtration). Say that a sequence of  mappings $\{Z_i:\Omega\to X\}_{i=0}^n$ ($X$-valued random variables)   is a $\mu$-martingale with respect to the filtration $\{\ms{F}_i\}_{i=0}^n$ if
\begin{equation}\label{eq:martingale condition}
\forall\, i\in\{1,\ldots,n\},\qquad \mathfrak{B}_\mu(Z_i|\ms{F}_{i-1})=Z_{i-1}.
\end{equation}
We will use below the following simple monotonicity property for martingales in barycentric metric space.

\begin{lemma} \label{lem:approximation} Fix $q\ge 1$ and $n\in \N$.  Let  $(X,d_X)$ be a $q$-barycentric metric space. Suppose that $\mu$ is a probability measure of full support on a finite set $\Omega$ and that  $\{Z_i:\Omega\to X\}_{i=0}^n$ is a $\mu$-martingale with respect to a filtration $\{\ms{F}_i\}_{i=0}^n$. Then,
\begin{equation} \label{eq:approximation}
\forall\, x\in X,\qquad \int_\Omega d_X(Z_0,x)^q\ud \mu\le \int_\Omega d_X(Z_1,x)^q\ud \mu\le \ldots\le  \int_\Omega d_X(Z_n,x)^q \ud\mu.
\end{equation}
\end{lemma}

\begin{proof} Fix $x\in X$ and $i\in\{0,1,\ldots,n-1\}$. Suppose that $\{A_j\}_{j=1}^k$ are the atoms of $\ms{F}_i$. For each $j\in \{1,\ldots,k\}$ and $\omega\in A_j$,
\begin{equation*} \label{eq:usepbarycinlemma}
\begin{split}
d_X(Z_i(\omega),x)^q  \stackrel{\eqref{eq:conditionalbarycenter}\wedge \eqref{eq:martingale condition}}{=} d_X\bigg(\mathfrak{B}\Big(\frac{1}{\mu(A_j)}  \sum_{a\in A_j}  \mu(a)\delta_{Z_{i+1}(a)} \Big),x \bigg)^q  \stackrel{\eqref{eq:pbarycentric}}{\leq} \frac{1}{\mu(A_j)} \sum_{a\in A_j} \mu(a) d_X(Z_{i+1}(a),x)^q.
\end{split}
\end{equation*}
Consequently,
\begin{equation*}\label{eq:on each atom}
\forall\, j\in \{1,\ldots,k\},\qquad \int_{A_j} d_X(Z_i,x)^q\ud \mu \le \sum_{a\in A_j} \mu(a) d_X(Z_{i+1}(a),x)^q.
\end{equation*}
As $\{A_j\}_{j=1}^k$ is a partition of $\Omega$, by summing this over $j\in \{1,\ldots,k\}$ we get that $\int_\Omega d_X(Z_i,x)^q\ud\mu\le \int_\Omega d_X(Z_{i+1},x)^q\ud\mu$.
\end{proof}

A key ingredient of the proof of Theorem~\ref{thm:main} is the following proposition, which is due to~\cite[Lemma~2.1]{MN13}; it  is analogous to an important inequality of Pisier~\cite{Pis75} for martingales with values in uniformly convex Banach spaces.

\begin{proposition} [Pisier's martingale cotype inequality in $q$-barycentric metric spaces] \label{prop:pisier} Fix $q>0$ and $n\in \N$. Let $(X,d_X)$ be a $q$-barycentric metric space with constant $\beta>0$. Suppose that $\mu$ is a probability measure of full support on a finite set $\Omega$ and that  $\{Z_i:\Omega\to X\}_{i=0}^n$ is a $\mu$-martingale with respect to a filtration $\{\emptyset,\Omega\}=\ms{F}_0\subset\ms{F}_1\subset\ldots\subset \ms{F}_n\subset 2^\Omega$. Then,
\begin{equation} \label{eq:pisier}
\forall\, x\in X,\qquad \int_\Omega d_X(Z_0,x)^q\ud\mu + \frac{1}{\beta^p} \sum_{i=1}^n \int_\Omega d_X(Z_i, Z_{i-1})^q\ud\mu \leq \int_\Omega d_X(Z_n,x)^q\ud\mu.
\end{equation}
\end{proposition}

\subsection{Cotype}

Fix $n\in\N$. Denote the uniform probability measure on $\{-1,1\}^n$ by $\mu$. For $i\in \{1,\ldots,n\}$ let  $\ms{F}_i\subset 2^{\{-1,1\}^n}$ be the $\sigma$-algebra that is generated by  the coordinate functions $\e_1,\ldots,\e_i:\{-1,1\}^n\to\{-1,1\}$. Write also $\ms{F}_0=\{\emptyset,\{-1,1\}^n\}$.

Suppose that $(X,d_X)$ is a metric space equipped with a barycenter map $\mathfrak{B}:\mathscr{P}_X^{<\infty}\to X$. For each $h:\{-1,1\}^n\to X$ we recursively construct a sequence of functions $\{\mathsf{E}_ih:\{-1,1\}^n\to X\}_{i=0}^n$ by setting $\mathsf{E}_n h=h$ and for every $i\in\{0,1,\ldots,n-1\}$,
\begin{equation} \label{eq:defineEi}
\forall\, \e\in \{-1,1\}^n,\qquad (\mathsf{E}_ih)(\e) \eqdef \mathfrak{B}_\mu\big( \mathsf{E}_{i+1}h \big| \mathscr{F}_i\big)(\e).
\end{equation}
By definition, $\{\mathsf{E}_ih\}_{i=0}^n$ is a $\mu$-martingale with respect to the filtration $\{\mathscr{F}_i\}_{i=0}^n$. In particular, for each $\e\in\{-1,1\}^n$ the value $(\mathsf{E}_i h)(\e)$ depends only on $\e_1,\ldots,\e_i$. We will therefore sometimes write $(\msf{E}_ih)(\e_1,\ldots,\e_i)$ instead of $(\msf{E}_ih)(\e)$.

Fix $m,n\in\N$ and $x\in\Z_m^n$. For every function $f:\Z_{m}^n \to X$, denote by $f_x:\{-1,1\}^n \to X$ the function that is given by
\begin{equation}\label{eq:translated f}
\forall\, \e\in\{-1,1\}^n,\qquad f_x(\e) = f(x+\e).
\end{equation}
 We record here the following simple identity for ease of later reference.

\begin{lemma} \label{lem:identityEi}
Fix $m,n\in\N$, $x\in\Z_m^n$ and a fuction $f:\Z_m^n\to X$. For every $i\in\{1,\ldots,n\}$, $j\in\{1,\ldots,i\}$ and $\e\in\{-1,1\}^n$, we have
\begin{equation} \label{eq:identityEi}
\big( \mathsf{E}_i f_{x-2\e_je_j}\big)(\e) = \big( \mathsf{E}_i f_x\big) (\e_1,\ldots,\e_{j-1},-\e_j,\e_{j+1},\ldots,\e_i).
\end{equation}
\end{lemma}

\begin{proof}
The proof proceeds  by reverse induction on $i\in \{0,\ldots,n\}$. For $i=n$, we have
\begin{equation*}
\big( \mathsf{E}_n f_{x-2\e_je_j}\big)(\e)  = f_{x-2\e_je_j}(\e) =f(x-2\e_je_j+\e)  = f_x(-2\e_je_j+\e)  = \big( \mathsf{E}_n f_x\big) (\e_1,\ldots,\e_{j-1},-\e_j,\e_{j+1},\ldots,\e_n).
\end{equation*}
For the inductive step, fix $i\in\{1,\ldots,n\}$ and notice that for every  $j\in\{1,\ldots,i-1\}$, we have
\begin{multline*}
\big( \mathsf{E}_{i-1}  f_{x-2\e_je_j}\big)(\e_1,\ldots,\e_{i-1})  = \mathfrak{B}\left(\frac{1}{2} \delta_{(\msf{E}_if_{x-2\e_je_j})(\e_1,\ldots,\e_{i-1},-1)}+\frac{1}{2} \delta_{(\msf{E}_if_{x-2\e_je_j})(\e_1,\ldots,\e_{i-1},1)} \right) \\  = \mathfrak{B}\left( \frac{1}{2}\delta_{(\msf{E}_if_x)(\e_1,\ldots,\e_{j-1},-\e_j,\e_{j+1},\ldots,\e_{i-1},-1)} + \frac{1}{2}\delta_{(\msf{E}_if_x)(\e_1,\ldots,\e_{j-1},-\e_j,\e_{j+1},\ldots,\e_{i-1},1)} \right) = (\msf{E}_{i-1}f_x)(\e_1,\ldots,\e_{j-1},-\e_j,\e_{j+1},\ldots,\e_{i-1}),
\end{multline*}
where the first and last equalities use the definition \eqref{eq:defineEi} and the middle equality uses the inductive hypothesis.
\end{proof}

\begin{proof}[Proof of Theorem \ref{thm:main}] Fix $m,n\in\N$ and a function  $f:\Z_{4m}^n\to X$. It will be notationally convenient to  prove the desired estimate~\eqref{eq:main} with $m$ replaced by $2m$, i.e., our goal is now to show that the following inequality holds true.
\begin{equation} \label{eq:main'}
\Bigg(\sum_{i=1}^n \sum_{x\in\Z_{4m}^n} d_X\big( f(x+2me_i),f(x)\big)^q\Bigg)^{\frac{1}{q}} \le \Big(4n^{\frac{1}{q}}+2\beta m\Big)\Bigg(\frac{1}{2^n} \sum_{\e\in\{-1,1\}^n} \sum_{x\in\Z_{4m}^n} d_X\big(f(x+\e),f(x)\big)^q\Bigg)^{\frac{1}{q}}.
\end{equation}

Recalling the notation~\eqref{eq:translated f} and using translation invariance on $\Z_{4m}^n$, for every $i\in \n$ and $\e\in \{-1,1\}^n$ we have
$$
\sum_{x\in\Z_{4m}^n} d_X\big( f_{x+2m\e_i e_i}(\e), f_x(\e)\big)^q=\sum_{x\in\Z_{4m}^n} d_X\big( f_{x+2me_i}(\e), f_x(\e)\big)^q=\sum_{x\in\Z_{4m}^n}  d_X\big( f(x+2me_i),f(x)\big)^q.
$$
Hence,
\begin{equation}\label{eq:use translatation invariance}
\Bigg(\sum_{i=1}^n\sum_{x\in\Z_{4m}^n}  d_X\big( f(x+2me_i),f(x)\big)^q\Bigg)^{\frac{1}{q}}=\Bigg(\frac{1}{2^n} \sum_{i=1}^n\sum_{\e\in\{-1,1\}^n}\sum_{x\in\Z_{4m}^n} d_X\big( f_{x+2m\e_i e_i}(\e), f_x(\e)\big)^q\Bigg)^{\frac{1}{q}}.
\end{equation}
By the triangle inequality in $(X,d_X)$, each of the summands in the right hand side of~\eqref{eq:use translatation invariance} can be bounded as follows.
\begin{equation}\label{eq:use triangle Ei}
d_X\big( f_{x+2m\e_i e_i}(\e), f_x(\e)\big)\le d_X\big( f_{x+2m\e_i e_i}(\e), (\msf{E}_if_{x+2m\e_i e_i})(\e) \big)  + d_X\big( (\msf{E}_i f_{x+2m\e_i e_i})(\e), (\msf{E}_if_x)(\e) \big) + d_X\big( f_{x}(\e), (\msf{E}_if_{x})(\e) \big).
\end{equation}
A substitution of~\eqref{eq:use triangle Ei} into~\eqref{eq:use translatation invariance} in combination with the triangle inequality in $L_q$ gives  the bound
\begin{equation}\label{eq:use triangle in Lq}
\begin{split}
\Bigg(\sum_{i=1}^n&\sum_{x\in\Z_{4m}^n}  d_X\big( f(x+2me_i),f(x)\big)^q\Bigg)^{\frac{1}{q}}\\  &\le 2\Bigg(\frac{1}{2^n} \sum_{i=1}^n\sum_{\e\in\{-1,1\}^n}\sum_{x\in\Z_{4m}^n} d_X\big( f_{x}(\e), (\msf{E}_if_{x})(\e) \big)^q\Bigg)^{\frac{1}{q}}+\Bigg(\frac{1}{2^n} \sum_{i=1}^n\sum_{\e\in\{-1,1\}^n}\sum_{x\in\Z_{4m}^n} d_X\big( (\msf{E}_i f_{x+2m\e_i e_i})(\e), (\msf{E}_if_x)(\e) \big)^q\Bigg)^{\frac{1}{q}},
\end{split}
\end{equation}
where we used the fact that, by translation invariance on $\Z_{4m}^n$ once more, we have
\begin{equation*}
\sum_{x\in \Z_{4m}^n} d_X\big(f_{x+2m\e_i e_i}(\e),(\msf{E}_if_{x+2m\e_i e_i})(\e)\big)^q = \sum_{x\in\Z_{4m}^n} d_X\big( f_x(\e), (\msf{E}_if_x)(\e)\big)^q.
\end{equation*}

To bound the first term on the right hand side of \eqref{eq:use triangle in Lq}, use the triangle inequality in $(X,d_X)$ to deduce that for every $i\in \n$, every $x\in\Z_{4m}^n$ and every $\e\in\{-1,1\}^n$, we have
\begin{equation*} \label{eq:boundfirst1}
d_X\big( f_{x}(\e), (\msf{E}_i f_x)(\e)\big)\leq  d_X\big( f(x+\e),f(x)\big) + d_X\big( (\msf{E}_i f_x)(\e), f(x) \big).
\end{equation*}
Hence, using the triangle inequality in $L_q$ we see that
\begin{equation}\label{eq:triangle tw term first term}
\begin{split}
\Bigg(\frac{1}{2^n} \sum_{i=1}^n&\sum_{\e\in\{-1,1\}^n}\sum_{x\in\Z_{4m}^n} d_X\big( f_{x}(\e), (\msf{E}_if_{x})(\e) \big)^q\Bigg)^{\frac{1}{q}}\\&\le
\Bigg(\frac{1}{2^n} \sum_{i=1}^n\sum_{\e\in\{-1,1\}^n}\sum_{x\in\Z_{4m}^n} d_X\big( f(x+\e),f(x)\big) ^q\Bigg)^{\frac{1}{q}}+\Bigg(\frac{1}{2^n} \sum_{i=1}^n \sum_{\e\in\{-1,1\}^n}\sum_{x\in\Z_{4m}^n} d_X\big( (\msf{E}_i f_x)(\e), f(x) \big)^q\Bigg)^{\frac{1}{q}}\\
&=n^{\frac{1}{q}}\Bigg(\frac{1}{2^n} \sum_{\e\in\{-1,1\}^n}\sum_{x\in\Z_{4m}^n} d_X\big( f(x+\e),f(x)\big) ^q\Bigg)^{\frac{1}{q}}+\Bigg(\frac{1}{2^n} \sum_{i=1}^n \sum_{\e\in\{-1,1\}^n}\sum_{x\in\Z_{4m}^n} d_X\big( (\msf{E}_i f_x)(\e), f(x) \big)^q\Bigg)^{\frac{1}{q}}.
\end{split}
\end{equation}
By Lemma \ref{lem:approximation}, for each fixed $i\in \n$ and  $x\in \Z_{4m}^n$,  we have
\begin{equation} \label{eq:boundfirst2}
\begin{split}
\frac{1}{2^n}\sum_{\e\in\{-1,1\}^n} d_X\big( (\msf{E}_i f_x)(\e), f(x)\big)^q \leq  \frac{1}{2^n}\sum_{\e\in\{-1,1\}^n} d_X  \big( (\msf{E}_nf_x)(\e), f(x)\big)^q  = \frac{1}{2^n}\sum_{\e\in\{-1,1\}^n} d_X\big( f(x+\e),f(x)\big)^q.
\end{split}
\end{equation}
By combining~\eqref{eq:triangle tw term first term} and~\eqref{eq:boundfirst2} we therefore conclude that
\begin{equation}\label{eq:first term done}
\Bigg(\frac{1}{2^n} \sum_{i=1}^n\sum_{\e\in\{-1,1\}^n}\sum_{x\in\Z_{4m}^n} d_X\big( f_{x}(\e), (\msf{E}_if_{x})(\e) \big)^q\Bigg)^{\frac{1}{q}}\le 2n^{\frac{1}{q}}\Bigg(\frac{1}{2^n} \sum_{\e\in\{-1,1\}^n}\sum_{x\in\Z_{4m}^n} d_X\big( f(x+\e),f(x)\big) ^q\Bigg)^{\frac{1}{q}}.
\end{equation}

To bound  the second term on the right hand side of~\eqref{eq:use triangle in Lq}, use the triangle inequality in $(X,d_X)$ to deduce that  for every $i\in \n$, every $\e\in \{-1,1\}^n$ and every $x\in \Z_{4m}^n$, we have
$$
d_X \big( (\msf{E}_i f_{x+2m\e_i e_i})(\e), (\msf{E}_if_x)(\e)\big) \leq  \sum_{s=1}^{m} d_X \big( (\msf{E}_i f_{x-2s\e_ie_i})(\e), (\msf{E}_if_{x-2(s-1)\e_ie_i})(\e) \big).
$$
By combining this with the triangle inequality in $L_q$, we see that
\begin{equation}\label{eq:m triangle}
\begin{split}
\Bigg(\frac{1}{2^n}\sum_{i=1}^n&\sum_{\e\in\{-1,1\}^n}\sum_{x\in\Z_{4m}^n} d_X\big( (\msf{E}_i f_{x+2m\e_i e_i})(\e), (\msf{E}_if_x)(\e) \big)^q\Bigg)^{\frac{1}{q}}\\&\le \sum_{s=1}^m \Bigg(\frac{1}{2^n}\sum_{i=1}^n\sum_{\e\in\{-1,1\}^n}\sum_{x\in\Z_{4m}^n} d_X \big( (\msf{E}_i f_{x-2s\e_ie_i})(\e), (\msf{E}_if_{x-2(s-1)\e_ie_i})(\e) \big)^q\Bigg)^{\frac{1}{q}}\\
&=m\Bigg(\frac{1}{2^n} \sum_{i=1}^n\sum_{\e\in\{-1,1\}^n}\sum_{x\in\Z_{4m}^n} d_X \big( (\msf{E}_i f_{x-2\e_ie_i})(\e), (\msf{E}_if_{x})(\e) \big)^q\Bigg)^{\frac{1}{q}}.
\end{split}
\end{equation}
where in the last step we used translation invariance on $\Z_{4m}^n$. Due to Lemma~\ref{lem:identityEi}, the estimate~\eqref{eq:m triangle} is the same as
\begin{equation}\label{eq:use lemma}
\begin{split}
\Bigg(\frac{1}{2^n}\sum_{i=1}^n&\sum_{\e\in\{-1,1\}^n}\sum_{x\in\Z_{4m}^n} d_X\big( (\msf{E}_i f_{x+2m\e_i e_i})(\e), (\msf{E}_if_x)(\e) \big)^q\Bigg)^{\frac{1}{q}}\\&\le m\Bigg(\frac{1}{2^n}\sum_{i=1}^n\sum_{\e\in\{-1,1\}^n}\sum_{x\in\Z_{4m}^n} d_X \big( (\msf{E}_i f_{x})(\e_1,\ldots,\e_{i-1},-\e_i), (\msf{E}_i f_x)(\e)\big)^q\Bigg)^{\frac{1}{q}}.
\end{split}
\end{equation}
Each summand in the right hand side of~\eqref{eq:use lemma} can be bounded using the triangle inequality in $(X,d_X)$ as follows.
\begin{equation}\label{use independent of i}
\begin{split}
d_X \big( (\msf{E}_i f_{x})(\e_1,\ldots,\e_{i-1},-\e_i),& (\msf{E}_i f_x)(\e)\big)\le d_X\big((\msf{E}_if_x)(\e_1,\ldots,\e_{i-1},-\e_i),(\msf{E}_{i-1}f_x)(\e)\big)+d_X\big((\msf{E}_{i-1}f_x)(\e),(\msf{E}_if_x)(\e)\big)\\
&= d_X\big((\msf{E}_if_x)(\e_1,\ldots,\e_{i-1},-\e_i),(\msf{E}_{i-1}f_x)(\e_1,\ldots,\e_{i-1},-\e_i)\big)
+d_X\big((\msf{E}_{i-1}f_x)(\e),(\msf{E}_if_x)(\e)\big),
\end{split}
\end{equation}
where the final step of~\eqref{use independent of i} holds because $(\msf{E}_{i-1}f_x)(\e)$ depends only on the variables $\e_1,\ldots,\e_{i-1}$. A substitution of~\eqref{use independent of i} into~\eqref{eq:use lemma}  together with an application of the triangle inequality in $L_q$ gives that
\begin{equation}\label{eq:use lemma + triangle}
\Bigg(\frac{1}{2^n}\sum_{i=1}^n\sum_{\e\in\{-1,1\}^n}\sum_{x\in\Z_{4m}^n} d_X\big( (\msf{E}_i f_{x+2m\e_i e_i})(\e), (\msf{E}_if_x)(\e) \big)^q\Bigg)^{\frac{1}{q}}\le 2m\Bigg(\frac{1}{2^n}\sum_{i=1}^n\sum_{\e\in\{-1,1\}^n}\sum_{x\in\Z_{4m}^n} d_X\big((\msf{E}_{i-1}f_x)(\e),(\msf{E}_if_x)(\e)\big)^q\Bigg)^{\frac{1}{q}}.
\end{equation}

For a fixed $x\in\Z_{4m}^n$, Proposition \ref{prop:pisier} applied to the martingale $\{\msf{E}_i f_x\}_{i=0}^n$ (on $\{-1,1\}^n$) gives the estimate
$$
\frac{1}{2^n}\sum_{i=1}^n\sum_{\e\in\{-1,1\}^n} d_X\big((\msf{E}_{i-1}f_x)(\e),(\msf{E}_if_x)(\e)\big)^q\le \frac{\beta^q}{2^n}\sum_{\e\in\{-1,1\}^n} d_X\big( f(x+\e), f(x)\big)^q.
$$
Due to~\eqref{eq:use lemma + triangle}, we therefore have
\begin{equation}\label{eq:second term all done}
\Bigg(\frac{1}{2^n}\sum_{i=1}^n\sum_{\e\in\{-1,1\}^n}\sum_{x\in\Z_{4m}^n} d_X\big( (\msf{E}_i f_{x+2m\e_i e_i})(\e), (\msf{E}_if_x)(\e) \big)^q\Bigg)^{\frac{1}{q}}\le 2\beta m \Bigg(\frac{1}{2^n}\sum_{\e\in\{-1,1\}^n}\sum_{x\in\Z_{4m}^n} d_X\big( f(x+\e), f(x)\big)^q\Bigg)^{\frac{1}{q}}.
\end{equation}
The desired estimate~\eqref{eq:main'} now follows by substituting~\eqref{eq:first term done} and~\eqref{eq:second term all done} into~\eqref{eq:use triangle in Lq}.
\end{proof}

\section{The definitions of metric cotype with $\ell_\infty$ edges and sign edges coincide}\label{rem:equivcotype}

Let $(X,d_X)$ be an arbitrary metric space. For every $q\in [1,\infty)$ and $m,n\in \N$ denote by  $\mathcal{C}_{q,m,n}^{\{-1,0,1\}}(X,d_X)$ the infimum over those $\mathcal{C}\in (0,\infty)$ such that for every $f:\Z_{2m}^n\to X$ we have
\begin{equation}\label{eq:MN13 def}
\Bigg(\sum_{i=1}^n\sum_{x\in \Z_{2m}^n} d_X\big(f(x+me_i),f(x)\big)^q \Bigg)^{\frac{1}{q}}\le \mathcal{C} m\Bigg(\frac{1}{3^n}\sum_{\e\in \{-1,0,1\}^n}\sum_{x\in \Z_{2m}^n} d_X\big(f(x+\e),f(x)\big)^q \Bigg)^{\frac{1}{q}}.
\end{equation}
Analogously, denote by $\mathcal{C}_{q,m,n}^{\{-1,1\}}(X,d_X)$ the infimum over those $\mathcal{C}_*\in (0,\infty)$ such that for every $f:\Z_{2m}^n\to X$ we have
\begin{equation}\label{eq:pm1 def}
\Bigg(\sum_{i=1}^n\sum_{x\in \Z_{2m}^n} d_X\big(f(x+me_i),f(x)\big)^q \Bigg)^{\frac{1}{q}}\le \mathcal{C}_* m \Bigg(\frac{1}{2^n}\sum_{\e\in \{-1,1\}^n}\sum_{x\in \Z_{2m}^n} d_X\big(f(x+\e),f(x)\big)^q \Bigg)^{\frac{1}{q}}.
\end{equation}

The estimate~\eqref{eq:MN13 def} corresponds to the original definition of metric cotype in~\cite{MN08}, while here we considered the variant~\eqref{eq:pm1 def}, in which the averaging in the right hand side is over the $2^n$ possible sign vectors in $\{-1,1\}^n$ rather than over the $3^n$ possible $\ell_\infty$-edges in $\{-1,0,1\}^n$. We will now show that these definitions are essentially equivalent, up to universal constant factors, by establishing the two estimates in~\eqref{eq:cotype equivalence} below, which hold for every metric space $(X,d_X)$, every $q\in [1,\infty)$ and every $m,n\in \N$. This confirms a prediction of~\cite[Section~5.2]{NS16}, where a special case was treated. Despite this fact, one should note that~\eqref{eq:MN13 def} is to some extent a more natural inequality than~\eqref{eq:pm1 def}, because the vectors $\{-1,0,1\}^n$ generate $\Z_{2m}^n$ while the vectors $\{-1,1\}^n$  do not. Therefore, unlike the right hand side of~\eqref{eq:pm1 def}, the right hand side of~\eqref{eq:MN13 def} corresponds to an "$L_q$ metric gradient" on a connected (Cayley) graph. We shall next establish that
\begin{equation}\label{eq:cotype equivalence}
\mathcal{C}_{q,m,n}^{\{-1,0,1\}}(X,d_X)\le 6^{\frac{1}{q}}\max_{k\in \n}\mathcal{C}_{q,m,k}^{\{-1,1\}}(X,d_X)\qquad\mathrm{and}\qquad \mathcal{C}_{q,2m,n}^{\{-1,1\}}(X,d_X)\le 2\mathcal{C}_{q,m,n}^{\{-1,0,1\}}(X,d_X).
\end{equation}
Thus, using the assumptions and notation of Theorem~\ref{thm:main}, its conclusion implies that also $\mathcal{C}_{q,2m,n}^{\{-1,0,1\}}(X,d_X)\lesssim n^{1/q}+\beta m$.

The proof of the first inequality in~\eqref{eq:cotype equivalence} is via the following simple reasoning; see also \cite[Lemma~2.7]{MN08}. Denote
\begin{equation}\label{eq:choose our C}
\mathcal{C}_*\eqdef \max_{k\in \n}\mathcal{C}_{q,m,k}^{\{-1,1\}}(X,d_X).
\end{equation}
If $A\subset \n$, then consider $\Z_{2m}^A$ to be a subset of $\Z_{2m}^n$, by identifying $y\in \Z_{2m}^A$ with $\sum_{i\in A} y_i e_i\in \Z_{2m}^n$. For $A\subset \n$ and $w\in \Z_{2m}^{\n\setminus A}$, define an auxiliary function $f_{A,w}:\Z_{2m}^A\to X$ by setting $f_{A,w}(y)=f(y+w)$ for every $y\in \Z_{2m}^A$. Recalling~\eqref{eq:choose our C}, an application of~\eqref{eq:pm1 def} to $f_{A,w}$ (with $n$ replaced by $|A|$)  gives the following estimate.
\begin{equation}\label{eq:for fixed A,w}
\begin{split}
\sum_{i\in A} \sum_{y\in \Z_{2m}^A}& d_X\big(f(y+w+me_i),f(y+w)\big)^q=\sum_{i\in A} \sum_{y\in \Z_{2m}^A} d_X\big(f_{A,w}(y+me_i),f_{A,w}(y)\big)^q\\&\le \frac{\mathcal{C}_*^q}{2^{|A|}} m^q \sum_{\d\in \{-1,1\}^A} \sum_{y\in \Z_{2m}^A} d_X\big(f_{A,w}(y+\d),f_{A,w}(y)\big)^q=\frac{\mathcal{C}_*^q}{2^{|A|}} m^q
\sum_{\d\in \{-1,1\}^A} \sum_{y\in \Z_{2m}^A} d_X\big(f(y+w+\d),f(y+w)\big)^q.
\end{split}
\end{equation}
It remains to observe that
\begin{equation}\label{eq:3n to 2n}
\begin{split}
\sum_{A\subset \n}\sum_{w\in \Z_{2m}^{\n\setminus A}}&2^{|A|}\sum_{i\in A} \sum_{y\in \Z_{2m}^A} d_X\big(f(y+w+me_i),f(y+w)\big)^q=
\sum_{i=1}^n \Bigg( \sum_{\substack{A\subseteq\{1,\ldots,n\}\\ i \in A}} 2^{|A|} \Bigg)\sum_{x\in \Z_{2m}^n} d_X\big(f(x+me_i),f(x)\big)^q\\&
=\sum_{i=1}^n \left( \sum_{k=1}^n {n-1\choose k-1}2^k\right)\sum_{x\in \Z_{2m}^n} d_X\big(f(x+me_i),f(x)\big)^q=
\frac{3^{n-1}}{2}\sum_{i=1}^n \sum_{x\in \Z_{2m}^n} d_X\big(f(x+me_i),f(x)\big)^q,
\end{split}
\end{equation}
and
\begin{equation}\label{eq:sum according to support}
\sum_{A\subset \n}\sum_{w\in \Z_{2m}^{\n\setminus A}}\sum_{\d\in \{-1,1\}^A} \sum_{y\in \Z_{2m}^A} d_X\big(f(y+w+\d),f(y+w)\big)^q=\sum_{\e\in \{-1,0,1\}^n}\sum_{x\in \Z_{2m}^n} d_X\big(f(x+\e),f(x)\big)^q.
\end{equation}
By multiplying~\eqref{eq:for fixed A,w} by $2^{|A|}$, summing the resulting bound over all $A\subset \n$ and $w\in \Z_{2m}^{\n\setminus A}$, and using the identities~\eqref{eq:3n to 2n} and~\eqref{eq:sum according to support}, we thus obtain the first inequality in~\eqref{eq:cotype equivalence}. Note that this deduction was entirely combinatorial and did not use the triangle inequality, but the proof of the second inequality in~\eqref{eq:cotype equivalence} does use (a modicum of) geometry.

For the second inequality in~\eqref{eq:cotype equivalence}, write $\mathcal{C}\eqdef\mathcal{C}_{q,m,n}^{\{-1,0,1\}}(X,d_X)$. Fix $f:\Z_{4m}^n\to X$. For each $\eta\in \{0,1\}^n$ define an auxiliary  function $\phi_\eta:\Z_{2m}^n\to X$ by setting $\phi_\eta(y)=f(2y+\eta)$ for every $y\in \Z_{2m}^n$ (for this, one should note that the mapping $y\mapsto 2y$ is well-defined as a mapping from $\Z_{2m}^n$ to $\Z_{4m}^n$). An application of~\eqref{eq:MN13 def} to each of the $2^n$ functions $\{\phi_{\eta}\}_{\eta\in \{0,1\}}$ gives
\begin{equation}\label{eq:to average eta}
\begin{split}
\sum_{i=1}^n \sum_{y\in \Z_{2m}^n}&d_X\big(f(2y+\eta+2me_i),f(2y+\eta)\big)^q=\sum_{i=1}^n \sum_{y\in \Z_{2m}^n}d_X\big(\phi_\eta(y+me_i),\phi_\eta(y)\big)^q\\&\le \frac{\mathcal{C}^q}{3^n} m^q \sum_{\e\in \{-1,0,1\}^n} \sum_{y\in \Z_{2m}^n} d_X\big(\phi_\eta(y+\e),\phi_\eta(y)\big)^q=\frac{\mathcal{C}^q}{3^n} m^q \sum_{\e\in \{-1,0,1\}^n} \sum_{y\in \Z_{2m}^n} d_X\big(f(2y+\eta+2\e),f(2y+\eta)\big)^q.
\end{split}
\end{equation}
The mapping $(y,\eta)\mapsto 2y +\eta$ is a bijection between
$\Z_{2m}^n \times \{0,1\}^n$ and $\Z_{4m}^n$.
So, by summing~\eqref{eq:to average eta} over $\eta\in \{0,1\}^n$ we get
\begin{equation}\label{eq:averaged cal C}
\sum_{i=1}^n \sum_{x\in \Z_{4m}^n}d_X\big(f(x+2me_i),f(x)\big)^q
\le \frac{\mathcal{C}^q}{3^n}\sum_{\e\in \{-1,0,1\}^n} \sum_{x\in \Z_{4m}^n}d_X\big(f(x+2\e),f(x)\big)^q.
\end{equation}

For every $A\subset \n$ and $v\in \Z_{2m}^n$, denote its restriction to $\Z_{2m}^A$ by $v_A=\sum_{i\in A} v_i e_i$. Observe that
\begin{equation}\label{eq:rewrite with restriction}
\sum_{\e\in \{-1,0,1\}^n} \sum_{x\in \Z_{4m}^n}d_X\big(f(x+2\e),f(x)\big)^q=\sum_{A\subset \n} \frac{1}{2^{n-|A|}}\sum_{\d\in \{-1,1\}^n}\sum_{x\in \Z_{4m}^n} d_X\big(f(x+2\d_A),f(x)\big)^q.
\end{equation}
For every fixed $(x,\d)\in \Z_{4m}^n\times \{-1,1\}^n$ and every fixed $A\subset \n$, the triangle inequality in $(X,d_X)$ combined with the  convexity of the mapping $t\mapsto t^q$ on $[0,\infty)$ gives that
\begin{equation}\label{eq:triangle on and off A}
\begin{split}
d_X\big(f(x+2\d_A),f(x)\big)^q&\le 2^{q-1}d_X\big(f(x+2\d_A),f(x+\d)\big)^q+2^{q-1}d_X\big(f(x+\d),f(x)\big)^q\\&= 2^{q-1}d_X\big(f(x+\d),f(x+\d+\d_A-\d_{\n\setminus A})\big)^q+2^{q-1}d_X\big(f(x+\d),f(x)\big)^q.
\end{split}
\end{equation}
Summing~\eqref{eq:triangle on and off A} over $x\in \Z_{4m}^n$, while keeping $\d\in \{-1,1\}^n$ and $A\subset \n$ fixed and using translation invariance, gives
\begin{equation}\label{eq:delta A}
\sum_{x\in \Z_{4m}^n}d_X\big(f(x+2\d_A),f(x)\big)^q\le 2^{q-1}\sum_{x\in \Z_{4m}^n}d_X\big(f(x),f(x+\d_A-\d_{\n\setminus A})\big)^q+2^{q-1}\sum_{x\in \Z_{4m}^n}d_X\big(f(x+\d),f(x)\big)^q.
\end{equation}
Since for a fixed $A\subset \n$, if $\d\in \{-1,1\}^n$ is equi-distributed over $\{-1,1\}^n$, then so is $\d_A-\d_{\n\setminus A}$, by~\eqref{eq:delta A} we have
$$
\sum_{\d\in \{-1,1\}^n}\sum_{x\in \Z_{4m}^n}d_X\big(f(x+2\d_A),f(x)\big)^q\le 2^q\sum_{\d\in \{-1,1\}^n}\sum_{x\in \Z_{4m}^n}d_X\big(f(x+\d),f(x)\big)^q.
$$
A substitution of this bound into~\eqref{eq:rewrite with restriction} gives that
\begin{equation}\label{eq:3n2n}
\begin{split}
\sum_{\e\in \{-1,0,1\}^n} \sum_{x\in \Z_{4m}^n}d_X\big(f(x+2\e),f(x)\big)^q&\le 2^q \Bigg(\sum_{k=1}^n \frac{\binom{n}{k}}{2^{n-k}}\Bigg) \sum_{\d\in \{-1,1\}^n}\sum_{x\in \Z_{4m}^n}d_X\big(f(x+\d),f(x)\big)^q\\&=\frac{2^q3^n}{2^n} \sum_{\e\in \{-1,1\}^n}\sum_{x\in \Z_{4m}^n}d_X\big(f(x+\e),f(x)\big)^q.
\end{split}
\end{equation}
By substituting~\eqref{eq:3n2n} into~\eqref{eq:averaged cal C}, we thus obtain the second inequality in~\eqref{eq:cotype equivalence}.

\section{Remarks on the barycentric condition}\label{sec:prelim}

 Fix $q,\beta>0$ and suppose that $(X,d_X)$ is a $q$-barycentric metric space with constant $\beta$, with respect to the barycenter map  $\mathfrak{B}:\ms{P}_X^{<\infty}\to X$.  If $a,b\in X$ are distinct points, then two applications of the definition~\eqref{eq:pbarycentric} to the probability measure  $\mu_{a,b}=(\d_a+\d_b)/2$   shows that $d_X(\mathfrak{B}(\mu_{a,b}),a)^q,d_X(\mathfrak{B}(\mu_{a,b}),b)^q\le d_X(a,b)^q/2$; for this, we do not need the second term in the left hand side of~\eqref{eq:pbarycentric}. Hence, $d_X(a,b)\le d_X(\mathfrak{B}(\mu_{a,b}),a)+d_X(\mathfrak{B}(\mu_{a,b}),b)\le 2^{1-1/q} d_X(a,b)$, and therefore $q\ge 1$. This is the reason why we assume $q\ge 1$ throughout, thus permitting the use of convexity in various steps.

We do not know if there exists a $q$-barycentric metric space for some $q\in [1,2)$, but we have the following statement.
\begin{proposition} \label{prop:laakso}
Suppose that $q\in [1,2)$ and $(X,d_X)$ is a non-singleton  $q$-barycentric metric space. Then there is a metric $\rho:X\times X\to [0,\infty)$ and $\theta\in (0,1)$ such that the metric spaces $(X,d_X)$ and $(X,\rho^\theta)$ are bi-Lipschitz equivalent to each other.
\end{proposition}

It follows in particular that a metric space $(X,d_X)$ as in Proposition~\ref{prop:laakso} cannot contain any rectifiable curve. Thus, any $q$-barycentric metric space which contains a geodesic segment must satisfy $q\ge 2$ (this is so, in particular, for geodesic metric spaces such as Alexandrov spaces). Note that there do exist barycentric metric spaces that do not contain any rectifiable curve, as exhibited by, say, the real line equipped with the metric given by $\sqrt{|s-t|}$ for all $s,t\in \R$, which is $4$-barycentric (since the real line itself is $2$-barycentric).

\begin{proof}[Proof of Proposition~\ref{prop:laakso}]
By a "de-snowflaking" result of Laakso~\cite[Theorem~7.2]{TW05}, if for no metric $\rho:X\times X\to [0,\infty)$ and $\theta\in (0,1)$  the desired conclusion holds, then for each $\e\in (0,1)$ there would be a metric $d_\e$ on the disjoint union $X\sqcup [-1,1]$ and (a scaling factor) $\lambda_\e\in (0,\infty)$ with the following properties. Firstly, $d_\e(s,t)=|s-t|$ for all $s,t\in [-1,1]$. Secondly,  $d_\e(x,y)=\lambda_\e d_X(x,y)$ for all $x,y\in X$. Finally, for all $s\in [-1,1]$ there exists $\sigma(s,\e)\in X$ such that $d_\e(\sigma(s,\e),s)\le \e$.

Suppose that $(X,d_X)$ is $q$-barycentric  with constant $\beta$ with respect to the barycenter map  $\mathfrak{B}:\ms{P}_X^{<\infty}\to X$. Denote
$$
\forall\, s\in [\e,1],\qquad \mu_\e^s=\frac12\d_{\sigma(-s,\e)}+\frac12 \d_{\sigma(s,\e)}\in  \ms{P}_X^{<\infty}.
$$
Since $q\ge 1$, we can use convexity to bound from below the second term in the left hand side of~\eqref{eq:pbarycentric} when $\mu=\mu_\e^s$ by
\begin{align}\label{eq:second term jensen}
\begin{split}
\Bigg(\int_X d_X\big(\mathfrak{B}(\mu_\e^s),y\big)^q \diff\mu_\e^s(y) \Bigg)^{\frac{1}{q}}&\ge \frac{d_X\big(\mathfrak{B}(\mu_\e^s),\sigma(-s,\e)\big)+d_X\big(\mathfrak{B}(\mu_\e^s),\sigma(s,\e)\big)}{2}\ge \frac{d_X\big(\sigma(s,\e),\sigma(-s,\e)\big)}{2}\\
&=\frac{d_\e\big(\sigma(s,\e),\sigma(-s,\e)\big)}{2\lambda_\e}\ge \frac{d_\e(-s,s)-d_\e\big(-s,\sigma(-s,\e)\big)-d_\e\big(s,\sigma(s,\e)\big)}{2\lambda_\e}\ge \frac{s-\e}{\lambda_\e}.
\end{split}
\end{align}
A substitution of~\eqref{eq:second term jensen} into two applications of~\eqref{eq:pbarycentric} gives the estimates
\begin{equation}\label{eq:use bary for -1}
d_X \big(\mathfrak{B}(\mu_\e^s),\sigma(-1,\e)\big)^q+\frac{(s-\e)^q}{(\beta\lambda_\e)^q}
\le \frac{d_X \big(\sigma(-1,\e),\sigma(-s,\e)\big)^q+d_X \big(\sigma(-1,\e),\sigma(s,\e)\big)^q}{2},
\end{equation}
and
\begin{equation}\label{eq:use bary for 1}
d_X \big(\mathfrak{B}(\mu_\e^s),\sigma(1,\e)\big)^q+\frac{(s-\e)^q}{(\beta\lambda_\e)^q}
\le \frac{d_X \big(\sigma(1,\e),\sigma(-s,\e)\big)^q+d_X \big(\sigma(1,\e),\sigma(s,\e)\big)^q}{2}.
\end{equation}
By averaging~\eqref{eq:use bary for -1} and~\eqref{eq:use bary for 1} and using convexity ($q\ge 1$) followed by the triangle inequality, we see that
\begin{align}\label{eq:average the bary uses geo}
\begin{split}
&\!\!\!\!\!\!\!\!\!\!\!\!\!\!\!\!\!\!\!\!\!\!\!\!\!\!\!\frac{d_X \big(\sigma(-1,\e),\sigma(-s,\e)\big)^q+d_X \big(\sigma(-1,\e),\sigma(s,\e)\big)^q+d_X \big(\sigma(1,\e),\sigma(-s,\e)\big)^q+d_X \big(\sigma(1,\e),\sigma(s,\e)\big)^q}{4}\\
&\ge \Bigg(\frac{d_X \big(\mathfrak{B}(\mu_\e^s),\sigma(-1,\e)\big)+d_X \big(\mathfrak{B}(\mu_\e^s),\sigma(1,\e)\big)}{2}\Bigg)^q+\frac{(s-\e)^q}{(\beta\lambda_\e)^q}\ge \frac{d_X \big(\sigma(-1,\e),\sigma(1,\e)\big)^q}{2^q}+\frac{(s-\e)^q}{(\beta\lambda_\e)^q}.
\end{split}
\end{align}
Next, we have
$$
d_X \big(\sigma(-1,\e),\sigma(1,\e)\big)=\frac{d_\e \big(\sigma(-1,\e),\sigma(1,\e)\big)}{\lambda_\e} \ge \frac{d_\e(-1,1) -d_\e \big(-1,\sigma(-1,\e)\big)-d_\e \big(1,\sigma(1,\e)\big)}{\lambda_\e}\ge \frac{2(1-\e)}{\lambda_\e}.
$$
Also,
$$
d_X \big(\sigma(-1,\e),\sigma(-s,\e)\big)=\frac{d_\e \big(\sigma(-1,\e),\sigma(-s,\e)\big)}{\lambda_\e}\le \frac{d_\e(-1,-s)+d_\e\big(-1,\sigma(-1,\e)\big)+d_\e(-s,\sigma(-s,\e)\big)}{\lambda_\e}\le \frac{1-s+2\e}{\lambda_\e}.
$$
Analogously, $d_X \big(\sigma(-1,\e),\sigma(s,\e)\big),d_X \big(\sigma(1,\e),\sigma(-s,\e)\big)\le (1+s+2\e)/\lambda_\e$ and $d_X \big(\sigma(1,\e),\sigma(s,\e)\big)\le (1-s+2\e)/\lambda_\e$. A substitution of these estimates into~\eqref{eq:average the bary uses geo} yields the  bound $2(1-\e)^q+2(s-\e)^q/\beta^q\le (1-s+2\e)^q+(1+s+2\e)^q$, which holds for every $0<\e\le s\le 1$. By taking  $\e\to 0$, we see that $2s^q/\beta^q\le (1+s)^q+(1-s)^q-2\lesssim s^2$. Hence necessarily $q\ge 2$.
\end{proof}

\bibliography{hadamard}
\bibliographystyle{abbrv}

\end{document}